\documentclass[10pt]{amsart}
\usepackage{etex}
\usepackage[french, english]{babel}
\usepackage{tikz}
\usetikzlibrary{matrix,arrows}
\usepackage[all,cmtip]{xy}
\usepackage{enumerate}
\usepackage{dsfont}
\usepackage{baskervald}
\usepackage{mathrsfs}
\usepackage{setspace}
\usepackage{longtable}
\usepackage{tabularx, tabu}
\usepackage{tikz}
\usepackage{amsfonts,amsthm,amsmath,amssymb,comment,stmaryrd,mathalfa,url,mathrsfs, yfonts}
\usepackage{graphicx}
\graphicspath{ {images/} }
\usepackage[section]{placeins}
\usepackage{microtype}
\usepackage{booktabs, calc, multirow}
\usepackage{array}
\usepackage[ocgcolorlinks,backref=page]{hyperref}
\hypersetup{citecolor=blue}
\usepackage{amscd}
\usepackage{makecell}
\usepackage{chngcntr}
\graphicspath{ {images/} }

\usepackage{breakcites}

\usetikzlibrary{external}
\usepackage[final]{changes}

\definechangesauthor[name={h}, color={black}]{h}
\definechangesauthor[name={c}, color={red}]{c}

\def\AA{\mathbb{A}}
\def\CC{\mathbb{C}}

\def\QQ{\mathbb{Q}}
\def\TT{\mathbb{T}}

\def\ZZ{\mathbb{Z}}
\def \lra{\longrightarrow}

\def \a{\alpha}

\def \g {\gamma}

\def \W {\mathcal{W}}

\def \l {\lambda}
\def \L {\Lambda}

\def \S {\mathcal{S}}

\def \Si {\Sigma}
\def \v {v}
\def \ii {\iota}

\def \vv {\text{{\fontfamily{lmr}v\selectfont}}}
\def \d {\mathfrak{d}}
\def \OO {\mathcal{O}}

\def \B {\mathcal{B}}
\def \K {U}

\def \KKK {\Gamma}

\def \CX {\mathscr{X}}

\def \unn {\underline}
\def \ov{\overline}

\def \wh{\widehat}

\def \nn {\mathfrak{n}}

\def \k {\kappa}
\def \Norm {N_{F/\QQ}}
\def\gothp{\mathfrak{p}}
\def\ps{\mathfrak{p}}
\def\gothq{\mathfrak{q}}

\DeclareMathOperator{\Mat}{Mat}
\DeclareMathOperator{\val}{val}

\DeclareMathOperator{\Spa}{Spa}
\DeclareMathOperator{\NP}{NP}

\DeclareMathOperator{\Inv}{Inv}

\DeclareMathOperator{\disc}{Disc}

\DeclareMathOperator{\GL}{GL}

\DeclareMathOperator{\Hom}{Hom}

\DeclareMathOperator{\Res}{Res}

\theoremstyle{plain}
\newtheorem{thm}{Theorem}[subsection]
\newtheorem*{thm*}{Theorem}
\newtheorem{lem}[thm]{Lemma}
\newtheorem{cor}[thm]{Corollary}
\newtheorem{con}[thm]{Conjecture}
\newtheorem*{con*}{Conjecture}
\newtheorem{prop}[thm]{Proposition}
\newtheorem*{prop*}{Proposition}
\theoremstyle{definition}
\newtheorem{defn}[thm]{Definition}

\theoremstyle{definition}
\newtheorem{rmrk}[thm]{Remark}
\newtheorem*{rmrk*}{Remark}
\newtheorem{warn}[thm]{Warning}
 
\newtheorem*{exmp*}{Example}
\newtheorem{obs}[thm]{Observation}
\newtheorem*{obs*}{Observation}
\newtheorem{nota}[thm]{Notation}
\newtheorem*{nota*}{Notation}
\newtheorem{num}[thm]{\unskip}
\newtheorem{question}[subsection]{Question}
\usepackage{fullpage}
\title{Slopes of overconvergent Hilbert modular forms}

\author{Christopher Birkbeck} 

\begin{document}
	
	\maketitle

	\begin{abstract}We give an explicit description of the matrix associated to the $U_p$ operator acting on spaces of overconvergent Hilbert modular forms over totally real fields. Using this, we compute slopes for weights in the centre and near the boundary of  weight space for certain real quadratic fields. \added[id=h]{Near the boundary of weight space we see that the slopes do not appear to be given by finite unions of arithmetic progressions but instead can be produced by a simple recipe from which we make a  conjecture on the structure of slopes. We also prove a lower bound on the Newton polygon of the $U_p$.}

	\end{abstract}

	\tableofcontents
	
	\section*{\textbf{ Introduction}} 
	\counterwithout{thm}{subsection}
	
	The idea of modular forms living in $p$-adic families began with Serre \cite{sere} who considered $p$-adic limits of compatible families of $q$-expansions of modular forms.  \added[id=h]{After this, Dwork \cite{Dworkp} studied subspaces of $p$-adic modular forms defined by imposing certain growth conditions and showed that on these spaces the $U_p$ operator is compact, thus giving us  a way to study these spaces in much more detail. Then Katz, in \cite{kertz}, showed that the spaces studied by Serre and Dwork could be defined in more geometric context. With this reformulation, Katz defined the spaces of overconvergent modular forms as subspaces of $p$-adic modular forms with appropriate \added[id=c]{growth} conditions. Moreover he showed that these spaces of overconvergent modular forms \added[id=c]{were} preserved by the action of the $U_p$ operator.}

	\added[id=h]{Then Hida, in a series of papers in the 1980's, defined and studied  subspaces of $p$-adic modular forms, topologically generated by the $p$-ordinary  eigenforms (which means that their $U_p$ eigenvalue is a $p$-adic unit). These forms are in fact overconvergent and have many remarkable properties, for example, Hida showed that} for a weight $k \geq 3$, the space of $p$-ordinary forms has rank depending only on $k$ modulo $p-1$ (or $2$ for $p=2$) \added[id=c]{and he used this to construct $p$-adic families of modular forms.} \added[id=h]{This  was then extended by Coleman--Mazur and Buzzard \cite{cmeig, buzeig} to finite slope eigenforms (which means the $U_p$ eigenvalue is not 0), by constructing geometric objects which they called eigencurves. These are rigid analytic varieties which parametrize all such modular forms of a fixed level and their points correspond to systems of eigenvalues for Hecke operators on finite slope overconvergent modular forms. More generally, Buzzard defined eigenvarieties and  created a ``machine'', which can be used to construct} eigenvarieties by inputting  a weight space and some suitable Banach modules together with an action of a Hecke algebra. \added[id=c]{More recently, Urban \cite{urb} and Hansen \cite{Hansen}  have used overconvergent cohomology groups to construct eigenvarieties associated to a large class of reductive groups. }

	Studying the geometry of these eigenvarieties is an active area of research and has many number theoretical applications, in particular,  Pottharst and Xiao in \cite{PX} have recently reduced the parity conjecture of Selmer ranks for modular forms to a statement about the geometry of certain eigenvarieties. Now, in order to study the geometry of these eigenvarieties, one can instead study the behaviour of the $p$-adic valuation of the $U_p$ eigenvalues (called the  {\it slopes}). Our goal is to compute slopes in many cases and make a precise conjecture on their structure and what it indicates about the geometry of the Hilbert eigenvariety. Our method relies on  working with modular forms defined on a totally definite quaternion algebra over a real quadratic field, which via the Jacquet-Langlands correspondence is enough to deduce results about the Hilbert eigenvariety (cf. Subsection \ref{jlcor}).
	
	In the case of modular forms over $\QQ$ this question has received a lot of attention recently, with a focus on studying  slopes of overconvergent modular forms as they move in $p$-adic families. To make this more precise, consider the Iwasawa algebra $\L=\ZZ_p \llbracket\ZZ_p^\times \rrbracket$  and let $\W$ be the associated rigid analytic space, which is called {\it weight space}. \added[id=h]{Elements of $\W(\CC_p)$ are identified with continuous homomorphisms $\ZZ_p^\times \to \CC_p^\times$, which are called {\it weights}. If we write $\ZZ_p^\times \cong H \times (1+q\ZZ_p)$ where $H$ is the torsion subgroup and where $q=p$ if $p$ is odd and $q=4$ for $p=2$, then taking a primitive Dirichlet character $\psi$ modulo $p^t$ and the character $z^k$ of $\ZZ_p^\times$ sending $z \mapsto z^k$ for $k  \in \ZZ$, we get an element of  weight space given by $z^k \psi$. The weights of the form $z^k$ are called {\it algebraic} and weights of  the form $z^k \psi$ are called {\it locally algebraic}. If we now take $\g$ a fixed topological generator of $1+q\ZZ_p$, and let $w(\kappa)=\k(\g)-1$ for $\kappa$ a weight, then the algebraic weights $z^k$ are in the region of  weight space such that $\val_p(w(\k)) \geq 1$ (for $p$ odd) called the {\it centre}, and the locally algebraic weights $z^k \psi$, with $\psi$ of conductor $p^t$ for $t \geq 2$ are in the region such that $\val_p(w(\k)) < 1$}\footnote{\added[id=h]{In fact one can show that they lie in $\val_p(w(\k))\leq \frac{1}{p-1}$ cf. \cite[Lemma 1.6]{bergp}.}} (again for $p$ odd\footnote{For $p=2$ the centre is where $\val_2(-) \geq 3$ and the boundary where $\val_2(-) < 3.$}).\added[id=h]{ Such locally algebraic weights are said to be  {\it near the  boundary}.} The reason we make such a distinction is that the behaviour of the slopes of the $U_p$ operator acting on weight $\k$ modular forms depends on where in  weight space $\k$ lives, as we shall see later. Lastly, we note that $\W \cong \bigsqcup_{\chi} \W_\chi$ where the $\chi$ run over characters of $H$ and $\W_\chi$ is corresponding component of  weight space.

	Over $\QQ$, the behaviour of the slopes of $U_p$ was first studied in \cite{gomaz} where they conjectured that if $k_1,k_2$ are large enough with $k_1 \equiv k_2 \mod p^n(p-1)$ for $n \geq \a$ for some rational number $\a$, then the dimension of the space of modular forms of weight $k_1$ and slope $\a$ should be the same as that of weight $k_2$ and slope  $\a$. Following this, Buzzard, Calegari, Jacobs, Kilford and Roe (among others) computed slopes of modular forms for weights both in the centre and boundary of weight space. In particular, in \cite{qasl} Buzzard computed slopes in many cases and was able to make precise conjectures about their behaviour.   Very little is known about the slopes near in centre of weight space and the geometry of the eigenvariety is expected to be more complicated. Results about slopes in this case can be found in \cite{buzcal,ghost}. In particular, Bergdall-Pollack have constructed a `ghost series' which conjecturally explains much of the behaviour of the slopes both near in centre and near the boundary of weight space.
	
	Near the boundary Buzzard--Kilford,  Jacobs and Roe were among the first to give evidence that  the sequence of slopes appear as a union of arithmetic sequences with same common difference. This  then implies that over the boundary of  weight space the eigencurve  looks like a countable union of annuli. For $p=2,3$ and trivial tame level this was proven by Buzzard-Kilford and Roe in \cite[Therem B]{BK}, \cite[Theorem 1]{roe}. For more details on the precise conjectures and their implications, see \cite{slmf}.  More generally, the recent work of Liu-Wan-Xiao and Wan-Xiao-Zhang  in \cite[Theorem 1.3, Theorem 1.5]{eovb}, \cite[Theorem C]{slbd} have proven similar results by working with quaternion algebras and using Chenevier's results mentioned above. In particular, they have defined integral models for these spaces of modular forms, and then used these to show that \added[id=c]{over the boundary of weight space the eigenvariety associated to a totally definite quaternion algebra over $\QQ$ is the disjoint union of countably many rigid analytic spaces which are finite and flat over an annulus. Furthermore, they show that the slopes are given by a union of arithmetic progressions with same common difference.}

	For overconvergent modular forms over $\QQ$ we have the following conjecture (which can be found in \cite[Conjecture 1.2]{eovb}, \cite[Conjecture 1.9]{bergp}) for the behaviour of the Newton polygon and the slopes of $U_p$.
	
	\begin{con*}
		For $\k$ a weight, let $s_1(\kappa),s_2(\kappa),\dots$ denote the slopes of the Newton polygon of\/z $U_p$ acting on the spaces of overconvergent modular forms of weight $\kappa$ and fixed level. Let $\NP_{\k}(U_p)$ be the Newton polygon of\/ $\det(1-XU_p)$. Then there exists an $r>0$ depending only on the component $\W_\chi$ of  weight space containing $\kappa$, such that
		
		\begin{enumerate}[$(a)$]
			\item For $\k \in \W_\chi$, such that $0 < \val_p(w(\k)) < r$, $\NP_{\k}(U_p)$ depends only on $\val_p(w(\k))$. Moreover, for weights in this component, the indices of the break points of the Newton polygon are independent of $\k$.
			\item  The sequence \added[id=h]{$\left ( s_i(\kappa)/\val_p(w(\kappa))\right )$} is a finite  union of arithmetic progressions, which is independent of $\kappa$ for $0 < \val_p(w(\kappa))<r$.
			
			\item Assuming $(a)$ above, the sequence of slopes $s_i(\k)$ are given by \added[id=h]{\[\bigcup_{i=0}^\infty \left ( S_{seed}+i\cdot \frac{\mid H \mid}{2} \right ),\]} where $S_{seed}$ is a fixed finite set\footnote{Here the notation is such that if $S$ is a sequence of slopes and $i \in \ZZ$, then we let $S+i$ denote the set, where we add $i$ to each slope in $S$. }, which only depends on the number of cusps of $X_0(M)$ (with $M$ the tame level) and the classical slopes in weight $2$ at different components of  weight space.\footnote{This was shown to follow from $(a)$ by Bergdall--Pollack in \cite[Theorem B]{bergp}.}
		\end{enumerate}
		
	\end{con*}

	Our goal here is to give computational evidence for a similar structure to the slopes of overconvergent Hilbert modular forms (in particular part $(c)$ above) and prove a lower bound for the Newton polygon of $U_p$. We compute explicit examples of sequences of slopes of the $U_p$ operator by using the Jacquet-Langlands correspondence. Throughout, we work with locally algebraic weights both in the centre of  weight space and near the boundary. The reason we only do this for locally algebraic weights is only for simplicity and these results can most certainly be extended to any weight.
	
	Our computations show that,  for $\k$ near the boundary of weight space (see Definition \ref{wtspdefn}), the slopes of classical Hilbert modular forms in weight $\kappa$ are generated analogously to part (c) above. In particular, our computations show that in some cases the slopes do not appear to be given as a finite union of arithmetic progressions. Moreover, the structure of the $U_p$ operator matrix in this case, suggests that the reason the slopes for modular forms over $\QQ$ are in arithmetic progression is due to the simpler nature of the $U_p$ operator in this case (specifically the rate at which the $p$-adic valuation of the entries increases as one goes down the rows  (cf. \ref{mult} )).
	
	\added[id=h]{Our methods also allow us to compute finite approximations $U_p(R,\k)$ to the infinite matrix of $U_p$ acting on overconvergent Hilbert modular forms of weight $\k$. In this case, since the $U_p$ operator is compact, one can prove there exists a function $f:\ZZ_{\geq 0} \to \ZZ_{\geq 0}$ (see Warning \ref{wow} for an explicit lower bound of this function) such that if the size of our approximation matrix is $R \times R$, then the first $f(R)$ smallest slopes of $U_p(R,\k)$ coincide with the first $f(R)$ smallest slopes of overconvergent Hilbert modular forms of weight $\k$. Unfortunately, the best bounds on $f$ that we have grow very slowly as $R$ increases; this means that, in practice, to prove that all of the approximated slopes we have computed are in fact slopes of overconvergent Hilbert modular forms (which we expect is the case), our $R$ needs to be much larger than we can currently compute with.}\footnote{For example, in some of our computations, we would need our approximation matrix to have  $R\sim 10^6$, although computations suggest that, in this case, we only need $R \sim 480$, but we cannot at this time prove this much stronger bound.}
	
	Our computations do however have much of the (conjectural) structure that one has over $\QQ$; meaning there is evidence that the overconvergent slopes can be `generated' by slopes appearing in the classical spaces of Hilbert modular forms of smallest classical weight (see Definition \ref{minwei}) analogous to what one sees over $\QQ$ (e.g. part (c) of the conjecture above). \added[id=h]{Furthermore, if $F$ is the totally real field over which our Hilbert modular forms are defined and $p=\prod_{\ps | p}\ps$ in $\OO_F$, then one not only has a $U_p$ operator, but also $U_{\ps}$ operators which satisfy $U_p=\prod_{\ps|p}U_\ps$.}
		
	\begin{nota*}
	\added[id=c]{For $V$ a compact operator we will let $\S(V)$ denote the set of pairs $(s,m)$ where $s$ represents the slope of an eigenvalue of $V$ and $m$ denotes the multiplicity with which it appears; we call such a pair an $sm$-pair. More generally, we call any subset of  $\QQ_{\geq 0} \times \ZZ_{\geq 1}$ a set of $sm$-pairs. Lastly, for $V \in \{U_p,U_{\gothp}\}$, we let $\S_{\k,r}(V)$ denote $\S(V)$ with $V$ acting in weight $(\k,r)$. }
	\end{nota*}	
	\added[id=c]{Note that if $V=U_{\ps}$ then this is not a compact operator on the full space of overconverget Hilbert modular forms, but one can restrict it to a subspace where it is compact, where again $\S_{\k,r}(V)$ makes sense. See \ref{cps} and \ref{ops} for more details.}	
		
	Now for $U_p$ or $U_\ps$  our computations suggest the following conjecture:

	\begin{con}\label{con1}
	\added[id=c]{	Let $[F:\QQ]=g$, $U$ be a sufficiently small level, $(\k,r)$ be a locally algebraic weight near the boundary. Then there exits a $T \in \ZZ_{\geq 0}$ and a finite set $B_{\k,r}(t,V)$ of $sm$-pairs for $t \in (\ZZ/T\ZZ)^g$  which only depend on which component $(\k,r)$ lies in, such that (after scaling the slopes of $V$ by $\val_p(w(\k))$)}
		\added[id=c]{\[\S_{\k,r}(V)=\bigcup_{t \in \ZZ_{\geq 0}^g} \left \{ B_{\k,r}(\ov{t},V) + l(t) \right \}\]} where:
		\begin{enumerate}
			\item  $\ov{t}$ is the (component-wise) reduction of\/ $t \mod T$.
			\item  For $t=(t_i)\in \ZZ_{ \geq 0}^g$,  $l(t)=\sum_{i=1}^g t_i$.
			\item  $\left \{ B_{\k,r}(\ov{t},V) + l(t) \right \}= \left \{ (a+l(t),b) : (a,b) \in  B_{\k,r}(\ov{t},V) \right \}.$ 
		\end{enumerate}
		
		 Moreover, on classical subspaces \[\S_{\k,r}(V \mid_{S_{\k,r}(U) })=\bigcup_{t \in S_{cl }} \left \{B_{\k,r}(\ov{t},V) + l(t) \right \}\]
		where $S_{cl}=\{t=(t_i) \in \ZZ^g | t_i \in [0,k_i-2]\}$ where $(k,r)$ is the algebraic part of $(\k,r)$ as defined in \ref{classweights}. (See Conjecture \ref{splitgencon} for an explicit description of the $T$ appearing above and \ref{con2} for a conjecture on what the $B_{\k,r}(\ov{t},V)$ are expected to be.).
	\end{con}

	\begin{rmrk*}
		Our computations near the boundary, for a fixed field $F$ and prime $p$, are limited to only changing the algebraic part of the weight and not the finite part, which means  $\val_p(w(\k))$ (which is defined in \ref{wtspdefn}) is always fixed. The reason for this is that  changing $\val_p(w(\k))$ requires working with more ramified characters and levels, which translates into much larger matrices than we can currently work with.
	\end{rmrk*}

	In Sections \ref{tab}, \ref{centrewtsp}, we collect some computations of slopes for locally algebraic weights  near the boundary and in the  centre of  weight space (i.e. with trivial character). Near the boundary we compute slopes in the cases when our chosen prime $p$ is split or inert in our totally real field. Furthermore, in the split case we also compute slopes for the $U_{\ps_i}$ and observe similar behaviour to that of $U_p$. \added[id=h]{In all cases, we observe that the sets $B_\k(t,V)$ above appear to depend only on the multiset  of slopes  of $U_p$ or $U_{\ps_i}$ acting on  Hilbert modular forms of smallest classical weight, with the $g$-tuple $t$ controlling which component of  weight space $\k$ lies in.} \added[id=h]{See Conjecture \ref{con2} for the precise formulation.}
	
	We also compute slopes for weights in the centre and observe that there is much less structure and the slopes are not all integers. This contrasts with what one sees in the $\Gamma_0$-regular case over $\QQ$, which suggests that in the Hilbert case, the structure in the centre should be much more complicated.

	Lastly, we prove a generalization of \cite[Theorem 4.8]{slbd}. Specifically we prove a lower bound for the Newton polygon of  $U_p$ on overconvergent Hilbert modular forms over any totally  real  field of even degree $g$ (see Proposition \ref{618}). In the case of real quadratic fields, the result is as follows: let $S_{2}^D(\K)$ denote the space of modular forms on $D$ of weight $[2,2]$  {\it including} the space of elements that factor through the reduced norm map (\added[id=h]{these correspond to reduced norm forms},  see \cite[Definition 3.7]{dembele}). 
	
	\begin{prop*}
		Let $F$ be a real quadratic field and  $\K=U_{0}(\nn p^s)$ be a sufficiently small level (see \ref{neat}) , $h= \dim(S_{2}^D(\K))$ and let $(\kappa,r)$ be any locally algebraic weight. Then the Newton Polygon of the action of\/ $U_p$ (appropriately normalized) on overconvergent Hilbert modular forms (over $F$) of level $\K$ weight $(\kappa,r)$ lies above the polygon with vertices
		\[(0,0),(h,0),(3h,2h), \dots , \left( \frac{i(i+1)h}{2},\frac{(i-1)i(i+1)h}{3} \right ), \dots.\]
	\end{prop*}
	
	\begin{rmrk*}
		Note this this is simply the polygon with $h$ slopes $0$, $2h$ slopes $1$, $3h$ slopes $2$ and so on.
	\end{rmrk*}

	% \subsection*{Organization} In Section 2 we recall some brief background on eigenvarieties and Chenevier's Interpolation theorem, which will be used in Section 5, where we prove the overconvergent Jacquet-Langlands correspondence for Hilbert modular forms. In Section 3-4 we briefly recall the construction of the eigenvarieties associated to Hilbert modular forms and quaternionic modular forms. Lastly, in Section 6-7 we investigate the applications of our results to computing slopes of overconvergent Hilbert modular forms, and we study the behaviour of these slopes. 

	\subsection*{Acknowledgements} The author would like to thank his supervisor Lassina Demb\'el\'e for his support and guidance. He would also like to thank Fabrizio Andreatta, David Hansen, Alan Lauder and Vincent Pilloni for interesting discussions and very useful suggestions. Lastly, this work is part of the authors thesis so I wish to thank my examiners Kevin Buzzard and David Loeffler, as well as the referee for their very useful comments and corrections. 
	\counterwithin{thm}{subsection}

	\section{\bf Notation and setup}
	 
	\begin{nota}\label{nota}
		
		\begin{enumerate}[(a)]
			\item Let $F$ be a totally field with $[F:\QQ]=g$ and let $p$ be a rational prime which is unramified in $F$.
			
			\item Let $\OO_F$ denote the ring of integers of $F$. For each finite place $\vv$ of $F$ let $F_\vv$ denote the completion of $F$ with respect to $\vv$ and $\OO_\vv$ the ring of integers of $F_\vv$. For an integral ideal $\nn$, let $F_\nn= \prod_{\vv|\nn} F_\vv$ and similarly let $\OO_\nn=\prod_{\vv \mid \nn} \OO_{\vv}.$ In particular, if we have $p\OO_F=\prod_{i=1}^f \gothp_i$, then let $\OO_p=\oplus_i \OO_{\gothp_i}=\OO_F \otimes \ZZ_p$.
			\item  Let $D$ be a totally definite quaternion algebra over $F$ with the finite part of the discriminant (denoted $\disc(D)$) being trivial. Let \added[id=c]{$D_f^\times:=\Res_{\OO_F/\ZZ}(D^\times)(\AA_f)$} and $\OO_D$ denote a fixed maximal order. \added[id=h]{Lastly, we fix an isomorphism $\OO_D \otimes_{\OO_F} \OO_p \cong M_2(\OO_p)$, which  induces an isomorphism $D_p:=D \otimes_{F} F_p \cong M_2(F_p)$.}
			
			\item Let $\Sigma$ be the set of all places of $F$, $\Sigma_p$ be the set of all finite places above $p$ and $\Sigma_\infty \subset \Sigma$ the set of all infinite places of $F$.  By abuse of notation we will often make no distinction between a finite place of $F$ and the associated prime ideal of $F$, but instead, when we want to think of the places as prime ideals we denote them by $\gothp$.

			\item \added[id=h]{Let  $\psi: \OO_p^{\times} \to \OO_{\CC_p}^{\times}$ denote a finite \added[id=c]{order} character of conductor $p^s$ for some $s \in \ZZ_{ \geq 1}$.}
			
			\item Let $\nn$ be an ideal of $\OO_F$ which is coprime to $p$ and \added[id=c]{fix a splitting of $\OO_D$ at the primes dividing $\nn$}. Let \[\K_0(\nn p^s):= \left \{\g \in (\OO_D \otimes \wh{\ZZ})^\times \mid \g \equiv \left ( \begin{smallmatrix} *&*\\ 0&*  \end{smallmatrix} \right) \mod \nn p^s \right \}. \] 
			\item \added[id=h]{Let $\Delta_s$ denote} the monoid of $\g= (\g_i)_{i \in \Sigma_p}= \left ( \begin{smallmatrix} a_i&b_i\\ c_i&d_i  \end{smallmatrix} \right)_{i \in \Sigma_p} \in M_2(\OO_p)$ such that \added[id=h]{$\det(\g_i) \ne 0$, $\pi_{\ps_i}^s \mid c_i$ and $\pi_{\ps_i} \nmid d_i$, where $\pi_{\ps_i}$ are fixed uniformisers at $\ps_i$. We denote $\Delta_1$ simply by $\Delta$.}
			
			\item \added[id=c]{For each $\vv \in \Sigma_\infty$, we have a field embedding $\ii_\vv$ of $F$ into $\CC$ given by $\vv$ and let $\ov{\QQ}$ denote the algebraic closure of $\QQ$ inside $\CC$ and we fix an algebraic closure $\ov{\QQ}_p$ of $\QQ_p$. Let $\ii: \CC \overset{\sim}\to \ov{\QQ}_p$ and define $inc_p:=\ii \circ id: \ov{\QQ} \to \ov{\QQ}_p$.} 
			\item \added[id=c]{Let $L$ be a complete extension of $\QQ_p$ containing the image of $\psi$ and  containing the compositum of the images of $F$ under $\iota \circ \ii_\vv$, for $\vv \in \Sigma_{\infty}$.}
			\item For variables $\unn{X}=(X_i)_{i=1}^g$ and $\unn{l}=(l_i)_{i} \in \ZZ_{\geq 0}^g$, we let $\unn{X}^{\unn{l}}=\prod_i X_i^{l_i}$.
			\item \added[id=h]{If $a,b \in \ZZ^g$ we will write $a+b$ (resp. $a-b$) for the $g$-tuple whose entries are given by $a_i+b_i$ (resp. $a_i-b_i$). Similarly, if $a \in \ZZ^g$ and $n \in \ZZ$ then by $a > n$ we mean that $a_i >n$ for $i \in \{1,\dots,g\}$.} 
			\item \added[id=h]{We fix throughout a labelling of the places in $\Sigma_{\infty}$ by elements in $\{1,\dots,g\}$. In particular, we identify $\ZZ^{\Sigma_{\infty}}$ with $\ZZ^g$.}  		
		\end{enumerate}

	\end{nota}
	
	\begin{rmrk}
		For computational purposes we will work with $F=\QQ(\sqrt d)$ where $d=5,13,17$  since these are real quadratic fields for which there exists a totally definite quaternion algebra $D/F$ with trivial discriminant and class number one\footnote{In fact $d=2,5,13,17$ are the only such examples, see \cite{voight}.}.
	\end{rmrk}

	\section{\textbf{Overconvergent quaternionic modular forms}}
	\subsection{\textbf{ The weight space}}\label{wei}
	In this subsection we define weights of Hilbert modular forms and the weight space. We will also define the boundary and centre of weight space. We begin with the classical definition of a weight of a Hilbert modular form over $F$.

	\begin{defn}\label{cw}
		Let $n \in \ZZ_{\geq 0}^{g}$ and $v \in \ZZ^{g}$ such that $n+2v=(r,\dots, r)$ for some $r \in \ZZ$. By abuse of notation we denote $(r,\dots,r)$ by $r$ for $r \in \ZZ$. Set $k= n+2$ and $w=v+n+1$. It follows from the above that all the entries of $k$ have the same parity and $k=2w-r$. We call the pair $(k,r) \in \ZZ_{\geq 2}^{g} \times \ZZ$ a {\it classical algebraic weight}.   Note that given $k$ (with all entries paritious and  $ \geq 2$) and $r$ we can recover $n,v,w$. In what follows we will move between both descriptions when convenient.  \added[id=h]{We will call $(k,r,n,v,w)$ satisfying the above the {\it weight tuple} associated to the {\it weight} $(k,r)$.}
	\end{defn}

	\begin{nota}\label{weightnota}
		If we take $k>2$ paritious, then we fix a choice of $w,n,\v,r$ as follows: let $k_0=\max_i\{k_i\}$ then set $\v=\left (\frac{k_0-k_i}{2} \right)_i$, $n=k-2$, $n_0=k_0-2$, $r=n_0$ and $w=n+\v+1$. We will only use this convention when doing computations at the end. 
	\end{nota}

	\begin{defn}\label{wtspdf}\label{wtspstr}Let $\TT=\Res_{\OO_F/\ZZ} \mathbb{G}_m$. We define $\W^{G}$ to be the rigid analytic space over $L$ associated to the completed group algebra $\OO_L\llbracket\TT(\ZZ_p)\times \ZZ_p^\times\rrbracket$. We call $\W^{G}$ the {\it{weight space for}} $G$. 
	\end{defn}

	One can show that, as rigid analytic spaces, \[\W^{G} \cong H^{\vee} \times B(1,1)^{g+1} \cong \bigsqcup_{\chi \in H^\vee} \W_\chi\]  where $H$ is the torsion subgroup of $\TT(\ZZ_p)\times \ZZ_p^\times$, $H^\vee$ is the character group of $H$ and $B(1,1)$ is the open ball of radius 1 around 1. The $\W_{\chi}$ are called the {\it components} of $\W$ and it is not hard to see that $\W^G(\CC_p)=\Hom_{cts}(\TT(\ZZ_p)\times \ZZ_p^\times, \CC_p^\times)$.

	\begin{nota}
		Elements of $\W^{G}(\CC_p)$ will be given by $v:\TT(\ZZ_p) \to \CC_p^{\times}$ and $r: \ZZ_p^{\times} \to \CC_p^{\times}$. Setting $n=-2v+r$ (\added[id=h]{where we are abusing notation and letting $r$ denote the map on $\TT(\ZZ_p)$ defined by \added[id=c]{$r \circ \Norm$}}) and $\k=n+2$, we will denote these weights as $(\k,r)$ and call $(\k,r,n,v,w)$ a weight tuple if $\k,r,n,v,w$ satisfy the same relations as in \ref{cw}.
	\end{nota}

	\begin{nota}\label{tei}
		\added[id=c]{Let us fix an isomorphism $\a:\TT(\ZZ_p) \times \ZZ_p^\times \simeq H \times \ZZ_p^{g+1}$ with $H$ the torsion subgroup of  $\TT(\ZZ_p) \times \ZZ_p^\times$.  Let $\tau$ denote the finite character of $\TT(\ZZ_p) \times \ZZ_p^\times$ sending elements to their image in $H$ under $\a$.}
	\end{nota}
	
	\begin{defn}\label{classweights} 
		\added[id=h]{A weight is called {\it locally algebraic} if it is the product of an algebraic weight $(k,r) \in \ZZ^g \times \ZZ$ and a finite character $\psi$, which we denote by $(\k_{\psi},r)$ or simply $(\k,r)$. For $(\k,r)$ an  locally algebraic weight we denote its algebraic part by $(k,r)$}. Lastly, a locally algebraic $(\k_\psi,r)$ is called {\it classical} if its algebraic part is a classical algebraic weight, \added[id=c]{i.e, if the algebraic part is such that $k_i \geq 2$ for all $i$.}  
		
	\end{defn}
	
	\begin{nota}
		Later, when working with real quadratic fields, we will sometimes denote locally algebraic weights $(\k_\psi,r)$ simply by $[k_1,k_2]\psi$. We will usually let $\psi$ be a character of \added[id=h]{$\OO_F$} of conductor $p^s$, viewed as a character of $\TT(\ZZ_p)$ (via strong approximation).
		
	\end{nota}	
	
	Next we define what it means for a weight to be in the centre or near the boundary of  weight space. To do this, we begin by thinking of weight space as an adic space. In this setting, one defines (following \cite{AIP3}) $\W_{adic}=\Spa(\Lambda_F,\Lambda_F)^{an}$,\added[id=c]{where $\Lambda_F=\OO_L\llbracket \TT(\ZZ_p) \times \ZZ_p^\times \rrbracket \cong \Lambda_F^0[H]$ with $\Lambda_F^0 \cong \OO_L \llbracket T_1, \dots, T_{g+1} \rrbracket$ and $\a:\TT(\ZZ_p) \simeq H \times \ZZ_p^{g+1}$ the same fixed isomorphism as before.}\footnote{Note that here, for consistency, we are defining  weight space over $\OO_L$, but with  more care one can work over $\ZZ_p$ which is more customary when discussing integral models, see \cite[Section 2]{AIP3}, but we do not need this here.} To see what the boundary should be, we can restrict to the trivial component of  weight space, i.e., $\W^0=\Spa(\Lambda_F^0,\Lambda_F^0)^{an}$, where $\Lambda_F^0$ has the $(p,T_1,\dots,T_{g+1})$-adic topology (here $p$ is assumed unramified in $F$).
	
	\begin{defn}\label{adicweights}
		Now define a continuous map (cf. \cite[Proposition 3.3.5]{scholze}) $c: \W^0 \lra [0,\infty]^{g+1}$ by $$x \longmapsto \left ( \frac{\log|T_1(\tilde x)|}{\log |p(\tilde x)|},\dots, \frac{\log|T_{\added[id=c]{g+1}}(\tilde x)|}{\log |p(\tilde x)|} \right ),$$ where $\tilde x$ is the maximal generalization of $x$. Note that $\log|T_i(\tilde x)|$ and $\log|p(\tilde x)|$ take values in $[-\infty,0)$ since the $T_i$ and $p$ are topologically nilpotent. From this it follows that $c(x)=(0,\dots,0)$ if and only if $|p(\tilde x)|=0$. Moreover, we note that we cannot have $x$ such that only some of the entries of $c(x)$ are zero, i.e., we cannot have $c(x)=(0,x_2,\dots,x_g)$ with $x_i \neq 0$. With this set-up, being  {\it near boundary of weight space} (in this component) is the same as having a point $x \in \W^0$ with $c(x)$ close to zero.
	\end{defn}
	
	As an example of weights that are near the boundary, we can take a classical weight $(k_\psi,r)$ where $\psi$ is a character sufficiently ramified at {\it every} prime above $p$. Now a natural question is, what if we take $\psi$ a character only ramified at {\it some} of the primes above $p$? It is not clear to the author if these points should morally be near the boundary of weight space or in the centre. For this reason we define a quasi-boundary (which contains the boundary) as follows:
	
	\begin{defn}\label{wtspdefn}
		Let $\k=(\k_1,\k_2,\dots, \k_g)$ be a weight on $\TT(\ZZ_p) \cong H \times \ZZ_p^g$. Fox a fixed choice of $h \in H$, let $\g_i$ be a  topological generator of the $i$-th copy of $\ZZ_p$ under the isomorphism above. Then, define $w(\k)=(\k_i(\g_i)-1) \in \CC_p^g$. In this way we  obtain a coordinate in  weight space for each of our weights. We also set $\val_p(w(\k))= \min_{i} \{\val_p(\k_i(\g_i)-1)\}$ and say that for an odd prime $p$ (resp. $p=2$), a weight $\k$ is near the {\it quasi-boundary} if $\val_p(w(\k)) < 1$ (resp. $\val_2(w(\k)) < 3$), otherwise we say it is in the {\it centre}. Note that this does not depend on the choice of isomorphism $\TT(\ZZ_p) \cong H \times \ZZ_p^g$.
	\end{defn}

	\begin{nota}\label{usefulnota}
		\added[id=h]{In what follows we make use of the following standard notation (see \cite[Section 9]{buzeig}). We first note that there is a surjection $\a:\Sigma_{\infty} \to \Sigma_p$. Then for $(a_\ps)_{\ps \in \Sigma_p}$ we let $(a_\vv)_{\vv \in \Sigma_{\infty}}$ be the tuple where for $\vv \in \Sigma_{\infty}$, we let $a_\vv$ denote $a_\ps$ for any $\vv \in \a^{-1}(\ps)$. This will be particularly useful when working with elements in $\Delta$ (see Notation \ref{nota}) which are indexed by elements of $\Sigma_p$.}
	\end{nota}
	
	In this setting, the spaces of overconvergent quaternionic modular forms are defined as follows:

	\begin{defn}\label{ocdefn} Let $L \langle \unn{X} \rangle$ denote the space of convergent power series in the variables $X_i$ for $i \in \{1,\dots,g\}$. The space of overconvergent quaternionic modular forms of weight $(\kappa_\psi,r)$, level $U_0(\nn p^s)$  is defined as the vector space of functions $$f: D^\times \backslash D_f^\times  \lra L\langle \unn{X} \rangle$$ such that $f(dg)=f(g)$ for all $d \in D^\times$ and $f(gu^{-1})\cdot u_p=f(g)$ for all $u \in U_0(\nn p^s)$ and $g \in D_f^\times$. Here the action of $\g= \left(\begin{smallmatrix} a&b\\ c&d \end{smallmatrix} \right) = \left( \left(\begin{smallmatrix} a_j&b_j\\ c_j&d_j \end{smallmatrix} \right) \right)  \in \Delta$ on $ L\langle \unn{X} \rangle$ is given by $$\prod_i X_i^{l_i} \cdot \g=\psi(d) \prod_i H(\g_i,X_i,l_i)$$ where $$H(\g_i,Z,t)=(a_i d_i-b_ic_i)^{v_i} (c_iZ+d_i)^{n_i} \left( \frac{a_iZ+b_i}{c_iZ+d_i} \right)^t,$$ $\Delta$ and $(\k_{\psi},r,n,v,w)$ is a weight tuple. We denote this space by $S_{\kappa_\psi,r}^{D,\dagger}(U_0(\nn p^s))$.
	\end{defn}

\begin{rmrk}
\added[id=h]{	To define the classical spaces of quaternionic modular forms one can take a classical locally algebraic weight $(\k_\psi,r)$ with associated weight tuple $(\k,r,n,v,w)$  and define $V_{n}(\unn{X})$ to be the space of polynomials in $\unn{X}=(X_i)$ such that the degree of $X_i$ is less than or equal to $n_i$ for $i \in \{1,\dots,g\}$. Then space of classical forms is defined by replacing $L\langle \unn{X} \rangle $ with $V_{n}(\unn{X})$  in Definition \ref{ocdefn} and using the same action of $\Delta$. We denote the resulting spaces by $S_{\kappa_\psi,r}^{D}(U_0(\nn p^s))$ or simply $S_{\kappa_\psi,r}(U_0(\nn p^s))$. }
\end{rmrk}
	
	\begin{rmrk}
		We use a slightly different convention for weight $[2,2]$ modular forms on $D$. It is customary to define $S_{2}(\K)$ as a quotient $S_{{2}}(\K) / \Inv(\K)$, where $\Inv(\K)$ is a subspace of forms that factor through the reduced norm map (see \cite[Section 1]{hidaon}). But for our purposes we do not quotient out by $\Inv(\K)$, so in weight $[2,2]$ are slightly different to what is usually defined.
	\end{rmrk}

	\begin{rmrk}\label{rk71}
		In order for the space of modular forms of weight $(\k_\psi,r)$ to be  non-trivial, one requires that $\psi(x)=\Norm(x)^r$ for all $x \in \OO_F^\times$, which we view as embedded in $\OO_p^\times$ in the usual way.
	\end{rmrk}

	\begin{rmrk}
		The above definition corresponds to working with overconvergent modular forms with radius of overconvergence $p^0=1$. Working with a fixed radius is not a problem, as one can show that the characteristic polynomial of  $U_p$ does not depend on this radius (\cite[Proposition 11.1]{buzeig}).
		
	\end{rmrk}
	
	\subsection{Relation to overconvergent Hilbert modular forms}\label{jlcor}

	We now explain why it is enough to study slopes of overconvergent quaternionic modular forms. The key result for this is the overconvergent Jacquet-Langlands correspondence, which, in this setting says:

	\begin{thm} 
		Let $D/F$ be a totally definite quaternion algebra  of  discriminant $\mathfrak{d}$ defined over a totally real field $F$. Let $p$ be a rational (unramified) prime and $\nn$ an integral ideal of $F$  such that  $p \nmid \nn\d$ and $(\nn,\d)=1$. Let $\CX_D(\nn p)$  be the eigenvariety of level $\nn p$ attached to quaternionic modular forms on $D$. Similarly, let $\CX_{\GL_2}(\nn \d p)$ denote the eigenvariety associated to cuspidal Hilbert modular forms of level $\nn \d p$ (with the associated moduli problem for this level being representable)  as constructed in \cite{AIP}. Then there is a closed immersion $\iota_D:\CX_{D}(\nn p) \hookrightarrow \CX_{\GL_2}(\nn \d p)$ which interpolates the classical Jacquet-Langlands correspondence. Moreover, when $[F:\QQ]$ is even, one can choose $D$ with $\d=1$ so that the above is an isomorphism between the corresponding eigenvarieties.
	\end{thm}
	
	\begin{proof}
		See \cite[Theorem 5.11]{me2}.
	\end{proof}
	
	Therefore, if we are interested in the geometry of $\CX_{\GL_2}(\nn \d p)$, for $F$ of even degree, then it is enough to study the slopes of overconvergent quaternionic modular forms. In general for $[F:\QQ]$ odd, one obtains a closed immersion from the quaternionic eigenvariety into the full Hilbert eigenvariety, so one can only study ``parts" of the full Hilbert eigenvariety, similar to the situation over $\QQ$ (cf. \cite{eovb}).
	
	\section{\textbf{The $U_p$ operator}}
	
	In this section we describe how to compute the $U_p$ operator matrix and prove a lower bound on its Newton polygon. The algorithms used to compute $U_p$ are very much inspired by \cite{dem,jacobs}. We note here that the results in this section apply to any totally definite quaternion algebra $D/F$ of class number one.  The fact that we work with a quaternion algebra that has class number one is simply to ease the exposition and computations, and one can most certainly work over any number field of even degree (or maybe even any degree) by adapting the work of Demb\'el\'e--Voight \cite{dembele}, but at the cost of increasing the computational complexity.
	\subsection{Explicit formulas for $U_p$}
	First note that, since $D$ is totally definite, then $D^\times \backslash D_f^\times / U$ is simply a finite number of points, which we call the {\it class number} of $(D,U)$ for $\K \subset \added[id=c]{\wh{\OO}_D^\times}$ an open compact subgroup. Moreover, since $D$ has class number one, then $D_f^\times= D^\times \wh{\OO}_D^\times$  and $D^\times \backslash D_f^\times = \OO_D^\times \backslash \wh{\OO}_D^\times$. Thus there is a bijection  $$D^\times \backslash D_f^\times / U \lra \OO_D^\times \backslash \wh{\OO}_D^\times/ U$$ and we can write $\wh{\OO}_D^\times =\coprod_{i=1}^{h} \OO_D^\times t_i U$ for $t_i$ suitable representatives. In what follows we will use the above decomposition to write an element $x \in D_f^\times$ as $du$ where $d \in D^\times, u \in \wh{\OO}_D^\times$. Setting $U=U_0(\nn p^s)$, we can use the bijection to  write $u=d't_i\g$ (for some $i$) where $d' \in \OO_D^\times$ and $\g \in U_0(\nn p^s)$. Now, following Demb\'el\'e \cite{dem}, we find the $t_i$ by observing that $$\OO_D^\times \backslash \wh{\OO}_D^\times/ U = \OO_D^\times \backslash \mathbb{P}^1(\OO_F/ \nn p^s)$$ where \added[id=c]{$\mathbb{P}^1(\OO_F/ \nn p^s)= \left \{ (a,b) \in (\OO_F/ \nn p^s)^2 \mid  \exists (\a,\beta) \in (\OO_F/ \nn p^s)^2 \text{ such that } \a a- \beta b =1 \right \}/ (\OO_F/ \nn p^s)^\times .$} We note that $$ \mathbb{P}^1(\OO_F/ \nn p^s) = \prod_{\gothq| \nn p^s} \mathbb{P}^1(\OO_F/\gothq^{e_\gothq}).$$ From this we can find the $t_i$ by simply picking a representative $(a,b)=(a_\gothq,b_\gothq)_{\gothq| \nn p^s} \in  \mathbb{P}^1(\OO_F/ \nn p^s) $ for each $\OO_D^\times$-orbit and then lifting this  to the element of $\wh{\OO}_D^\times$ which is $1$ at all places not diving $\nn p$ and, at the places dividing the level, we take $(\a_\gothq, \beta_\gothq) \in \left (\OO_F/\gothq^{e_\gothq}\right)^2$ such that  $a_\gothq \a_\gothq-b_\gothq \beta_\gothq =1$ Finally, we set $(t_i)_\gothq =  \left(\begin{smallmatrix} a_\gothq &b_\gothq\\ \beta_\gothq& \a_\gothq \end{smallmatrix} \right)$.

	\begin{nota}\label{neato}
		Let $\K \subset D_f^\times$ be an open compact subgroup and (by strong approximation) let $D_f^\times=\coprod_i D^\times t_i U$, where  $t_i \in D_f^\times$. For each $i$  set \[\Gamma^i(\K)=D^\times \cap t_i \K t_i^{-1},\]
		and note that \added[id=c]{$\overline{\Gamma^i(U)} = \Gamma^i(U)/(F^\times \cap \Gamma^i(\K))$} is {\it finite} by \cite[Lemma 7.1]{hidaon}. 
	\end{nota}
	
	\begin{defn}\label{neat} We say that $\K$ \textit{{sufficiently small}} if $\overline{\Gamma^i(\K)}$ is trivial for all $i$.
	\end{defn}

	\begin{defn}
		\added[id=h]{We say an open compact subgroup $U \in D_f^\times$ has {\it wild level} $\geq \pi^s$ if the projection $U \to D_p^\times$ is contained in $\Delta_s$. See \ref{nota} for the definition of $\Delta_s$ and $D_p$.}
	\end{defn}

	\begin{lem}\label{lem75}
		
		There is  an isomorphism
		
		\begin{equation} S_{\kappa_\psi,r}^{D,\dagger}(U) \overset{\sim}{\lra}\bigoplus_{i=1}^h L\langle \unn{X}	 \rangle^{\Gamma^i(U)} \label{triforce} \end{equation} given by sending\/ $f$ to\/ $(f(t_i))_i$.
		
	\end{lem}
	
	\begin{proof}Let $f \in  S_{\kappa_\psi,r}^{D,\dagger}(U) $. For  $g \in D_f^\times$ we can decompose it as $g=d t_i \g$ for some $i$, $d \in D^\times$ and $u \in \K$. Now the image of $g$ in (some)  $L\langle \unn{X}	 \rangle$ under $f$ is given by $f(g)=f(dt_i \g)=f(t_i \g)=f(t_i) \cdot u_p.$ Therefore it is enough to know where the $t_i$ are sent. But note that if $u \in \KKK^i(U)$, then $\g=t_i^{-1} d t_i$ for some $d \in D^\times$ and thus $f(t_i)=f(t_i t_i^{-1} d t_i)=f(t_i) \cdot u_p,$ from which we see that  the image must  be in $L\langle\unn{X}	 \rangle^{\Gamma^i(U)}$.

	\end{proof}

	\begin{rmrk}\label{applev}
		We note here that it is always possible to choose $\nn$ such that $U_0(\nn p^s)$ is sufficiently small \added[id=c]{(cf. \cite[Lemma 7.1]{hidaon}). Alternatively, we can take $U'$ being sufficiently small and consider level $U' \cap U_0(p^s)$. From now on we assume our level is sufficiently small.}
	\end{rmrk}

	Let $e$ denote the fundamental unit in $\OO_F^\times$ and let $(\k_{\psi},r)$ be a locally algebraic weight sending $e$ to $\Norm(e)^r$, which implies that $\Gamma^i(U)$ acts trivially on $ L\langle \unn{X}	 \rangle$ (by our sufficiently small assumption). Then from (\ref{triforce}) we have the following commutative diagram

	$$
	\xymatrix@C=4em@R=4em{
		S_{\k_\psi,r}^{D,\dagger}(U) \ar[r]^{\sim} \ar[d]^{U_\gothp} & \overset{h}{\underset{i=1}{\bigoplus}} L\langle \unn{X} \rangle \ar[d]_{\mathfrak{U}_{\gothp}}  \\
		S_{\k_\psi,r}^{D,\dagger}(U) \ar[r]^{\sim}  & \overset{h}{\underset{i=1}{\bigoplus}} L\langle \unn{X}	 \rangle.  }
	$$
	
	Therefore, in order to compute the action of $U_\gothp$, it is enough to compute $\mathfrak{U}_{\gothp}$. 
	
	\begin{defn}
		For $\ps | p$ let $\eta_\ps \in D_f^\times$ be the element which is the identity at all places different from $\ps$ and at $\ps$ it is the matrix $ \left ( \begin{smallmatrix} \pi_\ps&0\\ 0&1 \end{smallmatrix} \right)$, for $\pi_\ps$ a uniformiser of $F_\ps$. For each $\ps$ as above and $U=U_0(\nn p^s)$, we define the Hecke operators $T_\ps$ as the double coset operators given by $[\K \eta_\ps \K]$ and let $U_p=\prod_{\gothp \in \Sigma_p} U_{\ps}$ (recall that $p$ is unramified). \added[id=h]{Note that $T_\ps$ is independent of the choice of uniformiser.}  
		
	\end{defn}
	
	\begin{prop}
		For $U=U_0(\nn p^s)$, the double coset $[\K \eta_{\gothp} \K]$  can be written as  $$\coprod_{\a \in \OO_{\gothp}/ \pi_{\gothp}} U\left(\begin{smallmatrix} \pi_{\gothp} &0 \\ \a\pi_{\gothp}^{s_\gothp}& 1 \end{smallmatrix} \right).$$ 
	\end{prop}
	From this it follows that  the action of $U_{\gothp}$ is given by $ (f|U_{\gothp})(g)= \sum_{\a \in \OO_{\gothp}/ \pi_{\gothp}} f|_{u_\a}(g)$ for $g \in D_f^\times$, where $u_\a= \left(\begin{smallmatrix} \pi_{\gothp} &0 \\ \a\pi_{\gothp}^{s_\ps}& 1 \end{smallmatrix} \right)$. 
	
	\begin{defn}
		For each $t_i$ as above define $$\Theta(i,j):= \{\a \in \OO_{\gothp_i}/ \pi_{\gothp_i} \mid t_iu_\a^{-1}=dt_j\g_\a, \text{ for some } d \in D^\times, \g_\a \in U \}$$ and  let $T_{i,j}=\sum_{\beta \in \Theta(i,j)}  (\g_\beta u_\beta)_p $.  Here  $u_\beta= \left(\begin{smallmatrix} \pi_{\gothp} &0 \\ \beta\pi_{\gothp}^{s_\ps}& 1 \end{smallmatrix} \right).$

	\end{defn}

	\begin{prop}\label{matstruc}
		The matrices $(\g_\beta u_\beta)_p$ are in $ \left(\begin{smallmatrix} \pi \OO_p^\times &\OO_p\\ \pi^s \OO_p & \OO_p^\times\end{smallmatrix} \right)$ where\added[id=h]{ $U$ has wild level  $\geq \pi^s$.}
	\end{prop}	
	
	\begin{proof}
		The proof follows $mutatis$ $mutandis$ from the proof of \cite[Proposition 3.1]{eovb}.
	\end{proof}
	
	\begin{defn}\added[id=c]{Let $n \in \ZZ_{ \geq 0}$ and let $A$ be a $m \times m$ matrix with $m \equiv 0 \mod n$ or $m=\infty$. Then, by chosing a partion of the $m$ basis elements into disjoint subsets of size $n$, we can (after permuting the basis elements) consider $A$ as a  matrix partitioned into blocks of size $n \times n$. We call a matrix a $n \times n$-block matrix if we have made such a choice of partition of the basis elements} 
	\end{defn}

\begin{rmrk}\label{basis rem}
\added[id=c]{	In what follows, we will be looking at compact operators acting on $B=\oplus_{i=1}^h L \langle \unn{X} \rangle$. On this space we have a natural choice of basis given by the monomials $\unn{X}$ in each summand. Using this basis one can define the matrix attached to a compact operator and from now on, we will implicitly identify a compact operator with the matrix it defines.} 
	
\added[id=c]{Moreover, by choosing an ordering of our monomials we can think of our compact operators as an $h \times h$ block matrix as follows: let $e_i^j$ denote the $i$-th basis elements in the $j$-th summand in $B$, then we partition be basis elements by $\{e_0^1,\dots,e_0^h\} \cup \{e_1^1, \dots, e_1^h\} \cup \dots$, then using this we partition the matrix attached to a compact operator on $B$.}
\end{rmrk}

\begin{comment}

	\begin{prop}
		The action of\/ $\mathfrak{U}_{\gothp}$ is given by a $h \times h$-block matrix
		whose $(i,j)$-block is given by the infinite matrix of\/ $T_{i,j}$ on $L\langle \unn{X}	 \rangle $ \added[id=h]{for $i,j \in \{1,\dots,h\}$.}
	\end{prop} 
	
	\begin{proof}
		By Lemma \ref{lem75} we have that the action is given by
		
		\begin{align*}
		(f|U_{\gothp})(t_i)= \sum_{\a \in \OO_{\gothp_i}/ \pi_{\gothp_i}} f|_{u_\a}(t_i)  =\sum_{\a \in \OO_{\gothp_i}/ \pi_{\gothp_i}} f(t_i u_\a^{-1})\cdot (u_\a)_p =  \sum_{j=1}^h f(t_j)\cdot \left( \sum_{{\beta \in \Theta(i,j)}}  (\g_\beta u_\beta)_p \right)
		\end{align*}
		which gives the result.
	\end{proof}

\end{comment}	
	
	Now, recall that by Lemma \ref{lem75} we have that the action of $U_{\ps}$ is given by
	
	\begin{align*}
	(f|U_{\gothp})(t_i)= \sum_{\a \in \OO_{\gothp_i}/ \pi_{\gothp_i}} f|_{u_\a}(t_i)  =\sum_{\a \in \OO_{\gothp_i}/ \pi_{\gothp_i}} f(t_i u_\a^{-1})\cdot (u_\a)_p =  \sum_{j=1}^h f(t_j)\cdot T_{i,j}
	\end{align*}

	Similarly we can do all of the above for $U_p$. We now show how to explicitly write down the entries of matrix attached to $T_{i,j}$. \added[id=h]{For this we use the standard trick of using a generating function}\footnote{\added[id=h]{See \cite{jacobs,slbd,buzcal} for other places where such functions are used.}}. \added[id=h]{ Specifically, we want to find a power series in some number of variables, such that the entries of the matrix can be described in terms of the coefficients of this power series.}

	\begin{prop}\label{prop:7.1.8}The generating function for the $|_{\k_\psi,r}$ action of\/ $\g= \left(\begin{smallmatrix} a&b\\ c&d \end{smallmatrix} \right) = \left( \left(\begin{smallmatrix} a_j&b_j\\ c_j&d_j \end{smallmatrix} \right) \right) =(\g_j) \in \Delta$ with $(\k_{\psi},r)$ a locally algebraic weight  is given by $$\psi(d)\cdot \prod_{i=1}^g \frac{ \det(\g_{i})^{v_{i}}(c_{i} X_i+d_{i})^{n_{i}+1}}{(c_iX_i+d_i-a_iX_iZ_i-b_iZ_i)}$$ where $(\k_{\psi},r,n,v,w)$ is a weight tuple.

	\end{prop}
	
	\begin{proof}
		From the above, the action of $|_{\k_{\psi},r} \g$ on $ L\langle \unn{X} \rangle$ is given by 
		\begin{align*}
		\prod_i X_i^{l_i}|\g&=\psi(d)\prod_i  H(\g_i,X_i,l_i) =\sum_{\unn{m} \in \ZZ_{\geq 0}^g} a_{\unn{l},\unn{m}} \prod_{i} X_i^{m_i}
		\end{align*}  Now consider the formal sum $G(\unn{X},\unn{Z},\g)=\sum_{\unn{l},\unn{m}} a_{\unn{l},\unn{m}} \unn{Z}^{\unn{l}} \unn{X}^{\unn{m}} $, then	
		\begin{align*}
		&G(\unn{X},\unn{Z},\g)=\sum_{\unn{l}} \unn{Z}^{\unn{l}}  \sum_{\unn{m}} a_{\unn{l},\unn{m}} \unn{X}^{\unn{m}} \\
		&=\sum_{\unn{l}} \unn{Z}^{\unn{l}} \psi(d) \prod_i   H(\g_1,X_i,l_i) \\
		&=\psi(d) \prod_i \det(\g_i)^{v_i}  (c_iX+d_i)^{n_i}\sum_{l_i} Z_i^{l_i}  \left( \frac{a_iX_i+b_i}{c_iX_i+d_i} \right)^{l_i}  
		\end{align*} The result then follows by noting that $$\sum_{l_i} Z_i^{l_i}  \left( \frac{a_iX_i+b_i}{c_iX_i+d_i} \right)^{l_i}=\frac{1}{ 1-Z_i \left( \frac{a_iX_i+b_i}{c_iX_i+d_i} \right)}.$$ 
		
	\end{proof}
\added[id=c]{	From this we get an expression for $a_{\unn{l},\unn{m}}$. Then, after choosing a bijection from $\a:\ZZ^g \to \ZZ$ we define the matrix associated to the $|_{\k_{\psi},r}$ action as $(a_{\a(\unn{l}),\a(\unn{m})})_{\unn{l},\unn{m}}$. Note that, from the definition of $\Delta$, it follows that $|_{\k_{\psi},r}$ defined a compact operator on $ L\langle \unn{X} \rangle$.}

	\begin{cor}\label{prop712}  The coefficient of\/     $\prod_i X_i^{l_i}Z_i^{m_i}$ in $G(\unn{X},\unn{Z},\g)$ is $$\psi(d) \prod_{i=1}^g \det(\g_i)^{v_i}  C_{n_i}(\g_i,m_i,l_i)$$  where

		$$C_w\left ( \left(\begin{smallmatrix} a&b\\ c&d \end{smallmatrix} \right),x,y \right)= \sum_{t=0}^x {w-y \choose t} {y \choose x-t} a^{x-t}c^t d^{w-y-t}b^{y-x+t}. $$

	\end{cor}
	
	\begin{proof}
		The proof of this expression is a simple matter of expanding the power series, which is an un-illuminating computation. Similar results can be found in \cite[Appendix A]{jacobs}.
	\end{proof}
	
	\begin{num}\label{6.2.8}
		In order to write down a matrix for $U_p$ we need to choose a basis of $L\langle \unn{X} \rangle$. For computational purposes we will choose the one given by $\unn{X}^{\unn{l}}$ for $\unn{l} \in \ZZ_{\ge 0}^g$. Now in order to compute the finite approximations to the infinite matrix of $U_p$, we will also need to choose an ordering of this basis (cf. \ref{basis rem}), which is the same as choosing a  bijection $Bi:\ZZ_{\geq 0}^g \to \ZZ_{\geq 0}$. 	
	\end{num}
	\begin{nota}
		From now on, until the end of the section we make a fixed choice of bijection $Bi$ and if $m \in \ZZ_{ \geq 0}$ and $Bi^{-1}(m)=(m_i)_{i=1}^g$ then set $b(m)=\sum_{i} m_i$.
	\end{nota}

	\begin{rmrk}
		In what follows the choice of $Bi$ will only be for relevant for computational purposes. In particular, our theoretical results do not depend in an essential way in our choice of ordering. 
	\end{rmrk}

	It then follows from Corollary \ref{prop712} that:
	\begin{cor}\label{cor714}
		
		Let $\unn{x},\unn{y} \in \ZZ_{\ge 0}^g$ and $\g=  \left( \left(\begin{smallmatrix} \pi_j a_j &b_j\\ c_j\pi_j^{\added[id=h]{s}}&d_j \end{smallmatrix} \right) \right) \in \added[id=h]{\Delta_s.}$ Let $x=Bi(\unn{x})$ (similarly for $y$) and let $(\k_{\psi},r,n,v,w)$ be a weight tuple with $(\k_\psi,r)$ a locally algebraic weight. Then the $(x,y)$-th entry of the matrix representing the $\mid_{\k_\psi,r} \g$ action on $L\langle \unn{X}	 \rangle$  is given by\footnote{Recall that we are using Notation \ref{usefulnota}.} $$\Omega_{\k_\psi,r}(\g,x,y):= \psi(d)  \prod_{i=1}^g \pi_i^{x_i} \det(\g_i)^{v_i} d_i^{n_i} \frac{a_i^{x_i}}{d_i^{y_i}}  b_i^{y_i-x_i} C_{n_i}(\g_i,x_i,y_i)$$ 
		where $$C_{n_i}(\g_i,x_i,y_i)=\sum_{t=0}^{x_i} {n_i-y_i \choose t} {y_i \choose x_i-t} \left( \frac{b_i c_i}{a_i d_i} \right)^t  \pi_i^{t(\added[id=h]{s}-1)}.$$ 
	\end{cor}

	\begin{cor}\label{cor619}
		Let  $\g = \left( \left(\begin{smallmatrix} \pi_j a_j &b_j\\ c_j\pi_j^{s}&d_j \end{smallmatrix} \right) \right) \in \added[id=h]{\Delta_s}$ and $(\k_{\psi},r,n,v,w)$ a weight tuple with $(\k_{\psi},r)$ a locally algebraic weight. Then matrix for the weight $(\k_\psi,r)$ action of\/ $\g$ in $L \langle \unn{X} \rangle$ is such that  the $(x,y)$-th entry has $p$-adic valuation at least $$D+b(x) + \sum_{i=1}^g g(n_i,x_i,y_i)(s-1)$$ where:
		
		\begin{itemize}
			\item  $D=\val_p(\prod_i \det(\g_i)^v_i)$.
			\item   $(x_i),(y_i) \in \ZZ_{\geq 0}^g$, $x=Bi(\unn{x})$, $y=Bi(\unn{y})$.
			\item   $g(x,y,n)= \infty$ if  $x > n \geq y$, otherwise $$g(x,y,n) = \begin{cases} 
			x & \text{if\/ } y=0, \\
			0 & \text{if\/ } y \geq x, \\
			x-y & \text{if\/ } y < x.
			
			\end{cases}$$ (Note that having infinite $p$-adic valuation means that the entry of the matrix is zero.)
		\end{itemize}
	\end{cor}
	
	\begin{proof}
		This follows at once from Proposition \ref{matstruc} together with Corollary \ref{cor714} and noting that $g(n_i,x_i,y_i)$ is either $\infty$ or the first non-zero $t$ for which ${n_i-y_i \choose t}{y_i \choose x_i-t} \neq 0$.
	\end{proof}
	
Using the isomorphism $$ S_{\k_\psi,r}^{D,\dagger}(U) \overset{\sim}{\lra}   \bigoplus_{i=1}^h L\langle \unn{X}	 \rangle,$$ together with \ref{basis rem}, shows that the matrix associated to  $\mathfrak{U}_{p}$ can be thought of as an \added[id=c]{$h \times h$-block matrix}\footnote{Here $h$ is the class number of $(D,U)$.}. Moreover, since $T_{i,j}=\sum_{\beta \in \Theta(i,j)}  (\g_\beta u_\beta)_p $, the $(x,y)$-block of $\mathfrak{U}_{p}$ is given by $$M_{\k_\psi,r}(x,y):=(F_{i,j}^{\k_\psi,r}(x,y))_{i,j}$$ where $$F_{i,j}^{\k_\psi,r}(x,y):= \sum_{\beta \in \Theta(i,j)} \Omega_{\k_\psi,r}((\g_\beta u_\beta)_p,x,y )$$ and $i,j \in \{1,\dots,h\} $.

	 %\added[id=h]{whose $(i,j)$-block is given by the infinite matrix of the action of $T_{i,j}$} \added[id=h]{for $i,j \in \{1,\dots,h\}$. }Now, there is a natural basis\footnote{This is given by grouping the basis elements in each copy by degree.} of $\bigoplus_{i=1}^h L\langle  \unn{X}	 \rangle$ such that the matrix of $\mathfrak{U}_{p}$ becomes an infinite block matrix where each block has size $h \times h$.

	We now want to generalize \cite[Theorem 4.8]{slbd} to give a lower bound for the Newton polygon for the action of $U_p$. First we will normalize our Hecke operators.
	
	\begin{defn}\label{norma} \added[id=c]{For each prime ideal $\ps \in \Sigma_p$  let $\Si_{\ps}$ be the set of  $\vv \in \Sigma_\infty$ factoring through the projection $F_p \to F_{\ps}$. Let $\k$ be a locally algebraic weight with $(k,r,n,v,w)$ the associated weight tuple. Recall that $v \in \ZZ^g$ and that by \ref{nota} we can think of this as $v \in \ZZ^{\Sigma_{\infty}}$.  For each prime ideal $\ps \in \Si_{p}$ we define $v_{\ps}(\k)=\sum_{\vv \in \Sigma_{\ps}} v_\vv$. Using this we normalize our operators as follows: Let $U_p^0=\prod_{\ps \mid p} \pi_{\gothp}^{-v_\ps(\k)}U_p$ and $U_{\ps}^0=\pi_{\gothp}^{-v_\ps(\k)}U_\ps$.}
	\end{defn}			
	
	\begin{rmrk}
		Note that this has the effect of removing the terms $\det(\g_i)^{v_i}$ appearing in Corollary \ref{cor714}.
	\end{rmrk}

	\begin{prop}\label{618}
		Let $\K=U_0(\nn p^s)$ be a sufficiently small level and let $(\kappa,r)$ be any locally algebraic weight. For $n \in \ZZ_{ \geq 0}$ we let $s(n)$ denote the number of $m \in \ZZ_{ \geq 0}$ such that $b(m)=n$. Then the Newton Polygon of the action of\/ $U_p^0$ on $$S_{\k,r}^{D,\dagger}(U) \cong \bigoplus_{i=1}^h L\langle \unn{X}	 \rangle$$ lies above the polygon with slopes $\{0_{s(0)h},1_{s(1)h}, \dots, i_{s(i)h},\dots\}$  where $i_n$ means that $i$ appears $n$ times

	\end{prop}
	
	\begin{proof}
		We do this by giving a lower bound for \added[id=h]{the Hodge polygon (defined below) of $U_p^0$, which is always below the Newton polygon.}
		
		Now recall that the Hodge polygon is given by the lower convex hull of the vertices $(i,min_n)$, where $min_n$ is the minimal $p$-adic valuation of the determinants of all $n \times n$ minors\footnote{\added[id=h]{For more details on Hodge polygons see \cite[Section 4.7]{slbd} and \cite[Section 4.3]{kedlaya}.}}. Note that it clearly lies below the Newton polygon. Now Corollary \ref{cor619} gives that each $h \times h$ block $M_\k(x,y)$ is divisible by $p^{b(x)}$ Using this we can bound the Hodge polygon from below as follows: let $S=\{s_i\}:= \{0_{s(0)h},1_{s(1)h}, \dots, i_{s(i)h},\dots\}$ and let $\Sigma_i=\sum_{j \leq i} s_j$. Then from the above it is easy to see that the Hodge polygon is bounded from below by the convex hull of the points $(i, \Sigma_i)$.
	\end{proof}
	
	\begin{cor}
		In the case $g=2$, the Newton polygon is bounded below by the polygon with vertices given by	 $$(0,0),(h,0),(3h,2h), \dots , \left( \frac{i(i+1)h}{2},\frac{(i-1)i(i+1)h}{3} \right ), \dots.$$	
	\end{cor}
	
	\begin{proof}
		This follows at once from the proof of Proposition \ref{618}.
	\end{proof}

	\begin{section}{\textbf{Slopes near the boundary}}\label{tab}
		
		\setcounter{subsection}{1}	
		From now on we restrict to the case of $F$ being a real quadratic field.
		\begin{nota}
			In this section we will use convention given in \ref{weightnota} for our weights. Using this we will denote arbitrary locally algebraic weights as $\k$ and if we want to specify the character we will denote them as $[k_1,k_2]\psi$ where $k_i \in \ZZ_{\geq 2}$ and paritious.
			
		\end{nota}	
		
		\begin{nota}\label{bidef}
			In what follows we will choose the ordering of our basis (cf. \ref{6.2.8} ) given by	
			$$Bi(a,b)=\frac{(a+b+1)(a+b)}{2}+b$$ and $Bi^{-1}(m)=\left (m-\frac{t(t+1)}{2},\frac{t(t+3)}{2}-m \right)$ where $t=\left \lfloor \frac{-1+\sqrt{1+8m}}{2} \right \rfloor$. Lastly, for $Bi^{-1}(m)=(m_1,m_2)$ we note that $b(m)=m_1+m_2=t$.  
		\end{nota}	
		\begin{rmrk}
			\added[id=h]{We choose this particular ordering since computationally it appears to be the one for which  slopes stabilize quickest. This is most likely due to the appearance of the term $b(x)$ in Corollary \ref{cor619} which controls how the valuations of the entries of $U_p$ increase as we move down the rows.} 
		\end{rmrk}

	\added[id=h]{In this section we collect some computations of slopes of $U_p$, for $p$ a split or inert prime. The computations done below were done in \texttt{Magma} \cite{magma} and \texttt{Sage} \cite{sage}.}
	
	\begin{warn}
		 From now on all our Hecke operators will normalized as in \ref{norma}. For this reason we will denote them simply by $U_p, U_\ps$ instead of $U_p^0,U_{\ps}^0$.
	\end{warn}

		\begin{defn}\label{slopenot}
			If $K$ is a local field and $A \in \Mat_{n,n}(K)$ is a matrix, then we define the Newton polygon of $A$ to be the Newton polygon of the characteristic polynomial of $A$ and \added[id=c]{denote it $\NP(A)$. This naturally defines a set $\S(A)$ of $sm$-pairs, given by the slopes and multiplicities.}
		\end{defn}

		\begin{warn}\label{wow}
			\added[id=h]{When computing slopes of overconvergent Hilbert modular forms our strategy is to compute a finite  matrix $U_p(R,\k)$ which is a $R \times R$ approximation to the infinite matrix of $U_p$ acting on weight $\k$ overconvergent Hilbert modular forms and then compute the slopes of $U_p(R,\k)$, which we call the {\it approximated overconvergent slopes} . The fact that $U_p$ is compact means that we can find a function $f$ such that,  any vertex of $\NP(U_p(f(R),\k))$ of valuation less than $R$, will also be a vertex of $\NP(U_p(M,\k))$ for $M \geq f(R)$. So we can guarantee that the approximated overconvergent slopes are actually slopes of overconvergent Hilbert modular forms. Note that $f$ depends on the ordering of the basis of the matrix. If we use $Bi$ as in \ref{6.2.8} to order the basis, then $\lfloor \frac{b(R)}{h} \rfloor$bounds $f(R)$ from below, where  $b(R)=\left \lfloor \frac{-1+\sqrt{1+8R}}{2} \right \rfloor$.}\footnote{This is most likely not the optimal bound.}  Throughout this section, when we talk about overconvergent slopes, we mean approximated overconvergent slopes. 
			
			In the classical case we do not have this problem and all of the slopes we have computed are actually slopes of classical Hilbert modular forms. 
		\end{warn}

		\begin{subsection}{\textbf{ Split case}}

			Let $F=\QQ(\sqrt{13})$ and $p=3$. We will compute the slopes of $U_3$ on the space of modular forms of  $\K_0(9)$ for weights near the boundary. We note here that $U_0(9)$ is sufficiently small, which we checked computationally. In this case we find that $h=12$, where $h$ is the class number of $(D,U)$ with $D/F$ totally definite with $\disc(D)=1$.  We let $\psi_r$ be a continuous character of $\OO_p^\times$ of conductor $9$ such that $\psi_r(\a)=\Norm(\a)^r$ for $\a \in \OO_F^\times$. In the following table we list the slopes of classical Hilbert modular forms as a $sm$-pair $(s,m)$ where $s$ is the slope and $m$ its multiplicity. Note that since $p$ splits, we have $U_{p}=U_{\ps_1} U_{\ps_2}$. We also record here the classical slopes of $U_{\ps_1}, U_{\ps_2}$.
			
			\noindent		\begin{longtabu}{|c | c| X| }
				\Xhline{2 pt}
				
				Operator & Weight &  Classical Slopes \\ \Xhline{2 pt}
				$U_p$ & $[2,2]\psi_2$ & $ (  0 , 1 ) $,
				$ (  1/2 , 2 ) $,
				$ (  1 , 6 ) $,
				$ (  3/2 , 2 ) $,
				$ (  2 , 1 ) $
				\\ \hline
				$U_{\ps_1}$ & $[2,2]\psi_2$ &  $(  0 , 3 ) $, $ (  1/2 , 6 ) $, $ (  1 , 3 )$					
				\\ \hline
				$U_{\ps_2}$ & $[2,2]\psi_2$ &$(  0 , 3 ) $, $ (  1/2 , 6 ) $, $ (  1 , 3 )$					
				\\ \Xhline{2 pt}
				$U_p$ & $[2,4]\psi_2$ & (0, 1),
				$ (  1/2 , 2 ) $,
				$ (  1 , 7 ) $,
				$ (  3/2 , 4 ) $,
				$ (  2 , 8 ) $,
				$ (  5/2 , 4 ) $,
				$ (  3 , 7 ) $,
				$ (  7/2 , 2 ) $,
				$ (  4 , 1 ) $
				\\ \hline
				$U_{\ps_1}$ & $[2,4]\psi_2$ & $(0, 9)$, $(1/2, 18)$, $(1, 9)$
				\\ \hline
				$U_{\ps_2}$ & $[2,4]\psi_2$ & $ (  0 , 3 ) $,
				$ (  1/2 , 6 ) $,
				$ (  1 , 6 ) $,
				$ (  3/2 , 6 ) $,
				$ (  2 , 6 ) $,
				$ (  5/2 , 6 ) $,
				$ (  3 , 3 ) $	 
				\\ \Xhline{2 pt}		
				$U_p$ & $[2,6]\psi_2$ &$(0, 1)$,
				$ (  1/2 , 2 ) $,
				$ (  1 , 7 ) $,
				$ (  3/2 , 4 ) $,
				$ (  2 , 8 ) $,
				$ (  5/2 , 4 ) $,
				$ (  3 , 8 ) $,
				$ (  7/2 , 4 ) $,
				$ (  4 , 8 ) $,
				$ (  9/2 , 4 ) $,
				$ (  5 , 7 ) $,
				$ (  11/2 , 2 ) $,
				$ (  6 , 1 ) $
				\\ \hline
				$U_{\ps_1}$ & $[2,6]\psi_2$ &$(0, 15)$, $(1/2, 30)$, $(1, 15)$
				\\ \hline
				$U_{\ps_2}$ & $[2,6]\psi_2$ & $ (  0 , 3 ) $,
				$ (  1/2 , 6 ) $,
				$ (  1 , 6 ) $,
				$ (  3/2 , 6 ) $,
				$ (  2 , 6 ) $,
				$ (  5/2 , 6 ) $,
				$ (  3 , 6 ) $,
				$ (  7/2 , 6 ) $,
				$ (  4 , 6 ) $,
				$ (  9/2 , 6 ) $,
				$ (  5 , 3 ) $
				\\ \Xhline{2 pt}	
				
				$U_p$ & $[2,8]\psi_2$ &$(0, 1)$, $ (  1/2 , 2 ) $,
				$ (  1 , 7 ) $,
				$ (  3/2 , 4 ) $,
				$ (  2 , 8 ) $,
				$ (  5/2 , 4 ) $,
				$ (  3 , 8 ) $,
				$ (  7/2 , 4 ) $,
				$ (  4 , 8 ) $,
				$ (  9/2 , 4 ) $,
				$ (  5 , 8 ) $,
				$ (  11/2 , 4 ) $,
				$ (  6 , 8 ) $,
				$ (  13/2 , 4 ) $,
				$ (  7 , 7 ) $,
				$ (  15/2 , 2 ) $,
				$ (  8 , 1 ) $
				\\ \hline
				$U_{\ps_1}$ & $[2,8]\psi_2$ & $(0, 21)$, $(1/2, 42)$, $(1, 21)$
				\\ \hline
				$U_{\ps_2}$ & $[2,8]\psi_2$ & 
				$ (  0 , 3 ) $,
				$ (  1/2 , 6 ) $,
				$ (  1 , 6 ) $,
				$ (  3/2 , 6 ) $,
				$ (  2 , 6 ) $,
				$ (  5/2 , 6 ) $,
				$ (  3 , 6 ) $,
				$ (  7/2 , 6 ) $,
				$ (  4 , 6 ) $,
				$ (  9/2 , 6 ) $,
				$ (  5 , 6 ) $,
				$ (  11/2 , 6 ) $,
				$ (  6 , 6 ) $,
				$ (  13/2 , 6 ) $,
				$ (  7 , 3 ) $
				\\ \Xhline{2 pt}	
				$U_p$ & $[4,4]\psi_2$ &$ (  0 , 1 ) $,
				$ (  1/2 , 2 ) $,
				$ (  1 , 8 ) $,
				$ (  3/2 , 6 ) $,
				$ (  2 , 16 ) $,
				$ (  5/2 , 10 ) $,
				$ (  3 , 22 ) $,
				$ (  7/2 , 10 ) $,
				$ (  4 , 16 ) $,
				$ (  9/2 , 6 ) $,
				$ (  5 , 8 ) $,
				$ (  11/2 , 2 ) $,
				$ (  6 , 1 ) $		  	 
				\\ \hline
				$U_{\ps_1}$ & $[4,4]\psi_2$ &$(0, 9)$, $(1/2, 18)$, $(1, 18)$, $(3/2, 18)$, $(2, 18)$, $(5/2, 18)$, $(3, 9)$
				\\ \hline
				$U_{\ps_2}$ & $[4,4]\psi_2$ & $(0, 9)$, $(1/2, 18)$, $(1, 18)$, $(3/2, 18)$, $(2, 18)$, $(5/2, 18)$, $(3, 9)$
				\\ \Xhline{2 pt}	
				$U_p$ & $[3,3]\psi_1$ &$ (  0 , 1 ) $,
				$ (  1/2 , 2 ) $,
				$ (  1 , 8 ) $,
				$ (  3/2 , 6 ) $,
				$ (  2 , 14 ) $,
				$ (  5/2 , 6 ) $,
				$ (  3 , 8 ) $,
				$ (  7/2 , 2 ) $,
				$ (  4 , 1 ) $
				\\ \hline
				$U_{\ps_1}$ & $[3,3]\psi_1$ &$(0, 6)$, $(1/2, 12)$, $(1, 12)$, $(3/2, 12)$, $(2, 6)$
				\\ \hline
				$U_{\ps_2}$ & $[3,3]\psi_1$ & $(0, 6)$, $(1/2, 12)$, $(1, 12)$, $(3/2, 12)$, $(2, 6)$
				\\ \Xhline{2 pt} 
				$U_p$ & $[3,5]\psi_1$ &$(0, 1)$,
				$ (  1/2 , 2 ) $,
				$ (  1 , 8 ) $,
				$ (  3/2 , 6 ) $,
				$ (  2 , 15 ) $,
				$ (  5/2 , 8 ) $,
				$ (  3 , 16 ) $,
				$ (  7/2 , 8 ) $,
				$ (  4 , 15 ) $,
				$ (  9/2 , 6 ) $,
				$ (  5 , 8 ) $,
				$ (  11/2 , 2 ) $,
				$ (  6 , 1 ) $
				\\ \hline
				$U_{\ps_1}$ & $[3,5]\psi_1$ &$(0, 12)$, $(1/2, 24)$, $(1, 24)$, $(3/2, 24)$, $(2, 12)$
				\\ \hline
				$U_{\ps_2}$ & $[3,5]\psi_1$ &$ (  0 , 6 ) $,
				$ (  1/2 , 12 ) $,
				$ (  1 , 12 ) $,
				$ (  3/2 , 12 ) $,
				$ (  2 , 12 ) $,
				$ (  5/2 , 12 ) $,
				$ (  3 , 12 ) $,
				$ (  7/2 , 12 ) $,
				$ (  4 , 6 ) $
				\\ \Xhline{2 pt} 
				\caption{Split case, $p=3$, classical slopes.}
			\end{longtabu}
			
			\added[id=h]{Since our level is sufficiently small it follows that, for $[k_1,k_2] \neq [2,2]$ the dimension of the spaces of classical Hilbert modular forms of weight $[k_1,k_2]\psi_i$ and level $\K_0(9)$ is $12\cdot(k_1-1)\cdot(k_2-1)$ (this follows from translating \ref{lem75} to the classical setting)}. For weight $[2,2]$ the dimension of the classical space of cusp forms is $11$, but since in our notation we are including the reduced norm forms (which in this case contribute a $1$-dimensional subspace), we get a $12$ dimensional space.	With this one can easily see that (as long as we order our basis correctly\footnote{\added[id=h]{This means we choose an ordering such that the first $\dim(S_\k(U))$ basis elements form a basis of $S_{\k}(U)$ (note that this ordering may differ from the one given by $Bi$ as chosen in \ref{bidef})}.}) in the table above, the classical slopes $U_p$ in weight $\k$ are given by the slopes of $U_{*}(R,\k)$ for $R=(k_1-1)(k_2-1)12$ and  $* \in \{p,\gothp_1,\gothp_2\}$.

			\added[id=h]{We now compute the  overconvergent  slopes for the same set of weights as in the classical case but we only compute slopes of $U_p$ since the $U_{\ps_i}$ are not compact operators}. Our computations suggest that for a fixed $R$, the multiset of slopes of $U_p(R,\k)$ does not depend on the weight $\k$.\footnote{\added[id=h]{Specifically, we computed the slopes for many weights and they were always the same.}} For this reason, in the table below we only record the size of the matrix and  the slopes.
			
			\begin{longtabu}{|c| X| }
				\Xhline{2 pt}
				Matrix & Overconvergent Slopes \\ \Xhline{2 pt}

				$U_p(20 \cdot 12)$ &  		    
				
				$ (  0 , 1 ) $,
				$ (  1/2 , 2 ) $,
				$ (  1 , 8 ) $,
				$ (  3/2 , 6 ) $,
				$ (  2 , 16 ) $,
				$ (  5/2 , 10 ) $,
				$ (  3 , 24 ) $,
				$ (  7/2 , 14 ) $,
				$ (  4 , 32 ) $,
				$ (  9/2 , 18 ) $,
				$ (  5 , 39 ) $,
				$ (  11/2 , 20 ) $,
				$ (  6 , 35 ) $,
				$ (  13/2 , 10 ) $,
				$ (  7 , 5 ) $
				\\ \hline

				$U_p(25 \cdot 12)$ &  		    
				$ (  0 , 1 ) $,
				$ (  1/2 , 2 ) $,
				$ (  1 , 8 ) $,
				$ (  3/2 , 6 ) $,
				$ (  2 , 16 ) $,
				$ (  5/2 , 10 ) $,
				$ (  3 , 24 ) $,
				$ (  7/2 , 14 ) $,
				$ (  4 , 32 ) $,
				$ (  9/2 , 18 ) $,
				$ (  5 , 40 ) $,
				$ (  11/2 , 22 ) $,
				$ (  6 , 45 ) $,
				$ (  13/2 , 20 ) $,
				$ (  7 , 30 ) $,
				$ (  15/2 , 8 ) $,
				$ (  8 , 4 ) $
				\\ \hline

				$U_p(30 \cdot 12)$ & $ (  0 , 1 ) $,
				$ (  1/2 , 2 ) $,
				$ (  1 , 8 ) $,
				$ (  3/2 , 6 ) $,
				$ (  2 , 16 ) $,
				$ (  5/2 , 10 ) $,
				$ (  3 , 24 ) $,
				$ (  7/2 , 14 ) $,
				$ (  4 , 32 ) $,
				$ (  9/2 , 18 ) $,
				$ (  5 , 40 ) $,
				$ (  11/2 , 22 ) $,
				$ (  6 , 48 ) $,
				$ (  13/2 , 26 ) $,
				$ (  7 , 50 ) $,
				$ (  15/2 , 18 ) $,
				$ (  8 , 19 ) $,
				$ (  17/2 , 4 ) $,
				$ (  9 , 2 ) $
				\\ \Xhline{2 pt} 
				\caption{Split case, $p=3$, overconvergent slopes}		  		    		    
			\end{longtabu}
			
			\begin{obs}\label{obsplit}
				\begin{enumerate}[$(1)$]
					\item The slopes are appearing in arithmetic progression which is very similar to what we see over $\QQ$.  		  
					\item The multiplicities are not the same for each slope and are increasing, which is something that one does not see over $\QQ$ (cf. \cite[Theorem 1.5]{eovb}).
					
					\item In the table of classical slopes one can observe the effect of the Atkin-Lehner involution. We know from the basic theory of Hilbert modular forms, that we can define an involution $W_{p}$ (see \cite{lw93}) and this will send a Hilbert modular form of slope $\a$ in $S_{k_{\psi}}(\K_{0}(9))$ to a form of slope \[\val_p(\Norm(p)^{k_0-1-v_p(k)})-\a,\]in $S_{k_{\psi^{-1}}}(\K_0(9))$,  where $k_0=\max(k_1,k_2)$ . Now in our example, $\psi$ and $\psi^{-1}$ are in the same Galois orbit, so the slopes in weight $k_{\psi}$ and $k_{\psi^{-1}}$ will be the same. From which one can deduce that in the classical slopes above one should be able to pair up the slopes appearing in weight $[k_1,k_2]$ so that the slopes add up to $\val_p(\Norm(p)^{k_0-1-v_p(k)}),$ which is the case\footnote{The appearance of $v_p(k)$ is due to our normalization of $U_p$.}. Moreover, one can define Atkin-Lehner involutions $W_{\ps_i}$ for $i \in \{1,2\}$ and make similar observations in these cases.
					
					\item In the classical case we have not only computed the slopes of $U_p$, but those of $U_{\ps_1},U_{\ps_2}$. Here \added[id=c]{it appears that if we fix} $k_1$ and let $k_2$ grow, then the slopes of $U_{\ps_1}$ only increase in multiplicity, but we do not gain any new slopes. On the other hand, for $U_{\ps_2}$ we see that as $k_2$ increases the we gain new slopes. Moreover, if we instead fix $k_2$ and increase $k_1$ we see analogous behaviour.

				\end{enumerate}
			\end{obs}

			We have also done a similar computation in the case when $F=\QQ(\sqrt{17})$, $p=2$ and level $\K_0(8)$ (here $h=24$). In this case the behaviour is similar to the above, so we only record the slopes for a few small weights. The characters  $\psi_i$ appearing in  table $3$ are of conductor $8$.
			
			\begin{longtabu}{|c|  c| X |}
				\Xhline{2 pt}
				
				Operator & Weight &  Classical slopes \\ \Xhline{2 pt}
				$U_p$ & $[2,2]\psi_2$ & $ (  0 , 1 ) $,
				$ (  1/2 , 4 ) $,
				$ (  1 , 14 ) $,
				$ (  3/2 , 4 ) $,
				$ (  2 , 1 ) $
				\\ \hline
				$U_{\ps_1}$ & $[2,2]\psi_2$ & $(0, 4)$, $(1/2, 16)$, $(1, 4)$
				\\ \hline
				$U_{\ps_2}$ & $[2,2]\psi_2$ & $(0, 4)$, $(1/2, 16)$, $(1, 4)$
				\\ \Xhline{2 pt}
				$U_p$ & $[4,2]\psi_2$ & $(0, 1),$
				$(1/2, 4),$
				$(1, 15),$
				$(3/2, 8),$
				$(2, 16),$
				$(5/2, 8),$
				$(3, 15),$
				$(7/2, 4),$
				$(4, 1)$
				\\ \hline
				$U_{\ps_1}$ & $[4,2]\psi_2$ &$(0, 4),$ $(1/2, 16),$ $(1, 8),$ $(3/2, 16),$ $(2, 8),$ $(5/2, 16),$ $(3, 4)$
				\\ \hline
				$U_{\ps_2}$ & $[4,2]\psi_2$ & $(0, 12)$, $(1/2, 48)$, $(1, 12)$
				\\ \Xhline{2 pt}
				$U_p$ & $[3,3]\psi_1$ &$(0, 1),$ $(1/2, 4),$ $(1, 16),$ $(3/2, 12),$ $(2, 30),$ $(5/2, 12),$ $(3, 16),$ $(7/2, 4),$ $(4, 1)$
				\\ \hline
				$U_{\ps_1}$ & $[3,3]\psi_1$ & {\fontfamily{lmr} \selectfont (0, 8), (1/2, 32), (1, 16), (3/2, 32), (2, 8)}
				\\ \hline
				$U_{\ps_2}$ & $[3,3]\psi_1$ & {\fontfamily{lmr} \selectfont (0, 8), (1/2, 32), (1, 16), (3/2, 32), (2, 8)}
				\\ \Xhline{2 pt}
				Operator & Size &  Overconvergent slopes \\ \Xhline{2 pt}
				
				$U_p$ & $10 \cdot 24$ &$ (  0 , 1 ) $,
				$ (  1/2 , 4 ) $,
				$ (  1 , 16 ) $,
				$ (  3/2 , 12 ) $,
				$ (  2 , 32 ) $,
				$ (  5/2 , 20 ) $,
				$ (  3 , 48 ) $,
				$ (  7/2 , 28 ) $,
				$ (  4 , 59 ) $,
				$ (  9/2 , 16 ) $,
				$ (  5 , 4 ) $
					\\ \hline
				$U_p$ & $20 \cdot 24$ & {\fontfamily{lmr} \selectfont (0, 1), (1/2, 4), (1, 16), (3/2, 12), (2, 32), (5/2, 20), (3, 48), (7/2, 28), (4, 64), (9/2, 36), (5, 79), (11/2, 40), (6, 75), (13/2, 20), (7, 5)}
				\\ \Xhline{2 pt}
				\caption{Split case, $p=2$, classical and overconvergent slopes}
			\end{longtabu}

		\end{subsection}

		\subsection{\textbf{ Partial slopes}}\label{partslopes}
		Since we are working in the split case, we have $U_p=U_{\ps_1}U_{\ps_2}=U_{\ps_2}U_{\ps_1}$. Therefore a $U_p$ slope $\l_p$ can decomposed as a pair $(\l_{\ps_1},\l_{\ps_2})$ where $\l_{\ps_i}$ is a slope of $U_{\ps_i}$ and $\l_p=\l_{\ps_1}+\l_{\ps_{2}}$.

		\subsubsection{\textbf{Classical partial slopes}}\label{cps}
		For level $U_0(9)$ and weights $[4,4],[2,2],[2,4],[2,6],[2,8]$, we plot the pairs $(\l_{\ps_1},\l_{\ps_2})$ together with the multiplicity with which they appear. For the figures in this section, the horizontal axis denotes the slopes of $U_{\ps_1}$ and the vertical axis the slopes of $U_{\ps_2}$. The numbers in the grid represent the multiplicity with which this pair appears.
		
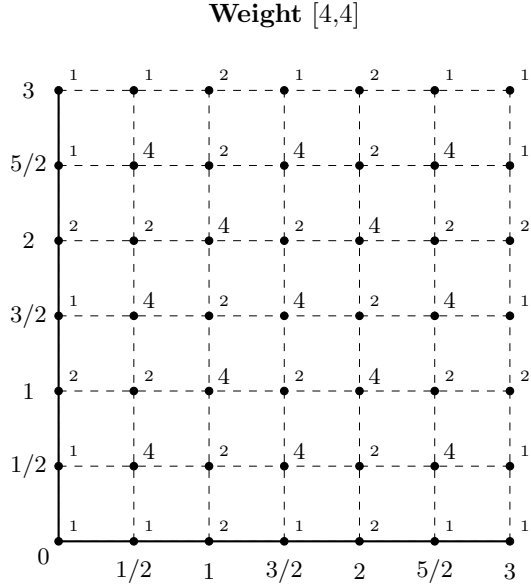
\begin{center}

			\begin{tikzpicture}

			\draw[step=1cm,very thin,dashed] (0,0) grid (6,6);	
			\node at (3,7) {\textbf{Weight} [4,4]};

			\node at (0+0.2,0+0.2) {\tiny 1};
			\node at (0+0.2,1+0.2) {\tiny 1};
			\node at (0+0.2,2+0.2) {\tiny 2};
			\node at (0+0.2,3+0.2) {\tiny 1};
			\node at (0+0.2,4+0.2) {\tiny 2};
			\node at (0+0.2,5+0.2) {\tiny 1};
			\node at (0+0.2,6+0.2) {\tiny 1};

			\node at (1+0.2,0+0.2) {\tiny 1};
			\node at (1+0.2,1+0.2) {\small 4};
			\node at (1+0.2,2+0.2) {\tiny 2};
			\node at (1+0.2,3+0.2) {\small 4};
			\node at (1+0.2,4+0.2) {\tiny 2};
			\node at (1+0.2,5+0.2) {\small 4};
			\node at (1+0.2,6+0.2) {\tiny 1};

			\node at (2+0.2,0+0.2) {\tiny 2};
			\node at (2+0.2,1+0.2) {\tiny 2};
			\node at (2+0.2,2+0.2) {\small 4};
			\node at (2+0.2,3+0.2) {\tiny 2};
			\node at (2+0.2,4+0.2) {\small 4};
			\node at (2+0.2,5+0.2) {\tiny 2};
			\node at (2+0.2,6+0.2) {\tiny 2};
			
			\node at (3+0.2,0+0.2) {\tiny 1};
			\node at (3+0.2,1+0.2) {\small 4};
			\node at (3+0.2,2+0.2) {\tiny 2};
			\node at (3+0.2,3+0.2) {\small 4};
			\node at (3+0.2,4+0.2) {\tiny 2};
			\node at (3+0.2,5+0.2) {\small 4};
			\node at (3+0.2,6+0.2) {\tiny 1};
			
			\node at (4+0.2,0+0.2) {\tiny 2};
			\node at (4+0.2,1+0.2) {\tiny 2};
			\node at (4+0.2,2+0.2) {\small 4};
			\node at (4+0.2,3+0.2) {\tiny 2};
			\node at (4+0.2,4+0.2) {\small 4};
			\node at (4+0.2,5+0.2) {\tiny 2};
			\node at (4+0.2,6+0.2) {\tiny 2};
			
			\node at (5+0.2,0+0.2) {\tiny 1};
			\node at (5+0.2,1+0.2) {\small 4};
			\node at (5+0.2,2+0.2) {\tiny 2};
			\node at (5+0.2,3+0.2) {\small 4};
			\node at (5+0.2,4+0.2) {\tiny 2};
			\node at (5+0.2,5+0.2) {\small 4};
			\node at (5+0.2,6+0.2) {\tiny 1};
			
			\node at (6+0.2,0+0.2) {\tiny 1};
			\node at (6+0.2,1+0.2) {\tiny 1};
			\node at (6+0.2,2+0.2) {\tiny 2};
			\node at (6+0.2,3+0.2) {\tiny 1};
			\node at (6+0.2,4+0.2) {\tiny 2};
			\node at (6+0.2,5+0.2) {\tiny 1};
			\node at (6+0.2,6+0.2) {\tiny 1};

			\draw[step=1cm,thick] (0,0) -- (0,6);
			
			\draw[step=1cm,thick] (0,0) -- (6,0);
			\foreach \x in {0,1,2,3,4,5,6}
			\foreach \y in {0,1,2,3,4,5,6}
			\draw [fill] (\x,\y) circle [radius=0.05];
			
			\draw (-0.4,1) node {\small 1/2};
			\draw (-0.4,2) node {\small 1};
			\draw (-0.4,3) node {\small 3/2};
			\draw (-0.4,4) node {\small 2};
			\draw (-0.4,5) node {\small 5/2};
			\draw (-0.4,6) node {\small 3};
			
			\draw (1,-0.4) node {\small 1/2};
			\draw (2,-0.4) node {\small 1};
			\draw (3,-0.4) node {\small 3/2};
			\draw (4,-0.4) node {\small 2};
			\draw (5,-0.4) node {\small 5/2};
			\draw (6,-0.4) node {\small 3};
			
			\draw (-0.2,0-0.2) node {\small 0};

			\end{tikzpicture}
			
		\end{center}
		
		\newpage
		
		
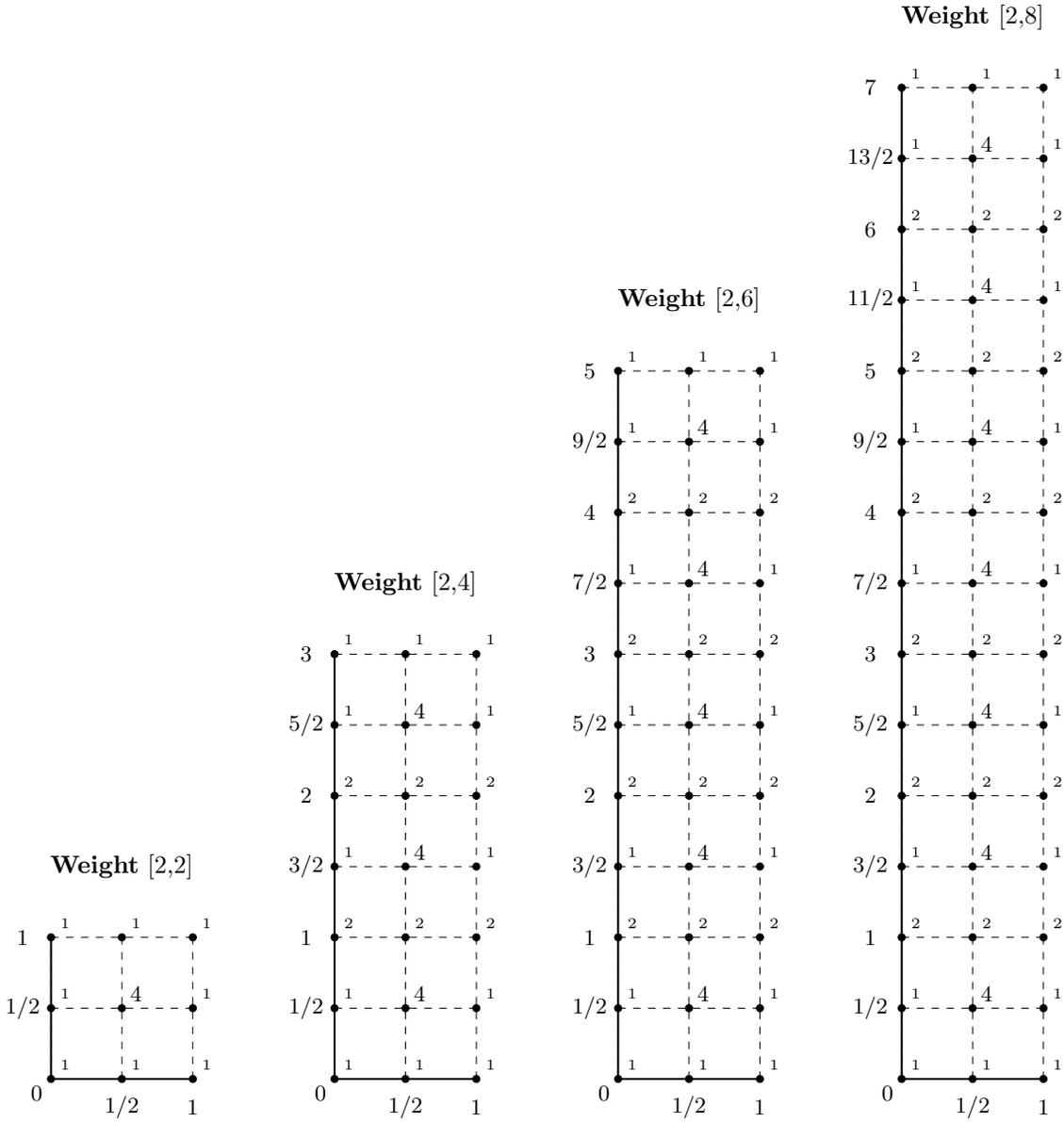
\begin{tikzpicture}

		\foreach \x in {0,1,2}
		\foreach \y in {0,1,2}
		\draw [fill] (\x,\y) circle [radius=0.05];

		\node at (1+0.2,1+0.2) {\small 4};
		\node at (0+0.2,0+0.2) {\tiny 1};
		\node at (1+0.2,0+0.2) {\tiny 1};
		\node at (0+0.2,1+0.2) {\tiny 1};
		\node at (1+0.2,2+0.2) {\tiny 1};
		\node at (2+0.2,1+0.2) {\tiny 1};
		\node at (2+0.2,2+0.2) {\tiny 1};
		\node at (2+0.2,0+0.2) {\tiny 1};
		\node at (0+0.2,2+0.2) {\tiny 1};
		
		\node at (1,3) {\textbf{Weight} [2,2]};
		
		\draw[step=1cm,thick] (0,0) -- (0,2);
		
		\draw[step=1cm,thick] (0,0) -- (2,0);
		\draw[step=1cm,very thin,dashed] (0,0) grid (2,2);

		\draw (1,-0.4) node {\small 1/2};
		\draw (2,-0.4) node {\small 1};
		
		\draw (-0.2,0-0.2) node {\small 0};
		\draw (-0.4,1) node {\small 1/2};
		\draw (-0.4,2) node {\small 1};
		
		\hspace{4cm}
		\draw[step=1cm,very thin,dashed] (0,0) grid (2,6);	
		\draw (-0.4,1) node {\small 1/2};
		\draw (-0.4,2) node {\small 1};
		\draw (-0.4,3) node {\small 3/2};
		\draw (-0.4,4) node {\small 2};
		\draw (-0.4,5) node {\small 5/2};
		\draw (-0.4,6) node {\small 3};

		\draw (1,-0.4) node {\small 1/2};
		\draw (2,-0.4) node {\small 1};
		
		\draw (-0.2,0-0.2) node {\small 0};
		
		\node at (0+0.2,0+0.2) {\tiny 1};
		\node at (0+0.2,1+0.2) {\tiny 1};
		\node at (0+0.2,2+0.2) {\tiny 2};
		\node at (0+0.2,3+0.2) {\tiny 1};
		\node at (0+0.2,4+0.2) {\tiny 2};
		\node at (0+0.2,5+0.2) {\tiny 1};
		\node at (0+0.2,6+0.2) {\tiny 1};

		\node at (1+0.2,0+0.2) {\tiny 1};
		\node at (1+0.2,1+0.2) {\small 4};
		\node at (1+0.2,2+0.2) {\tiny 2};
		\node at (1+0.2,3+0.2) {\small 4};
		\node at (1+0.2,4+0.2) {\tiny 2};
		\node at (1+0.2,5+0.2) {\small 4};
		\node at (1+0.2,6+0.2) {\tiny 1};

		\node at (2+0.2,0+0.2) {\tiny 1};
		\node at (2+0.2,1+0.2) {\tiny 1};
		\node at (2+0.2,2+0.2) {\tiny 2};
		\node at (2+0.2,3+0.2) {\tiny 1};
		\node at (2+0.2,4+0.2) {\tiny 2};
		\node at (2+0.2,5+0.2) {\tiny 1};
		\node at (2+0.2,6+0.2) {\tiny 1};

		\node at (1,7) {\textbf{Weight} [2,4]};

		\draw[step=1cm,thick] (0,0) -- (0,6);
		
		\draw[step=1cm,thick] (0,0) -- (2,0);
		\foreach \x in {0,1,2}
		\foreach \y in {0,1,2,3,4,5,6}
		\draw [fill] (\x,\y) circle [radius=0.05];

		\hspace{4cm}
		\draw[step=1cm,very thin,dashed] (0,0) grid (2,10);	
		\node at (1,11) {\textbf{Weight} [2,6]};

		\node at (0+0.2,0+0.2) {\tiny 1};
		\node at (0+0.2,1+0.2) {\tiny 1};
		\node at (0+0.2,2+0.2) {\tiny 2};
		\node at (0+0.2,3+0.2) {\tiny 1};
		\node at (0+0.2,4+0.2) {\tiny 2};
		\node at (0+0.2,5+0.2) {\tiny 1};
		\node at (0+0.2,6+0.2) {\tiny 2};
		\node at (0+0.2,7+0.2) {\tiny 1};
		\node at (0+0.2,8+0.2) {\tiny 2};
		\node at (0+0.2,9+0.2) {\tiny 1};
		\node at (0+0.2,10+0.2) {\tiny 1};

		\node at (1+0.2,0+0.2) {\tiny 1};
		\node at (1+0.2,1+0.2) {\small 4};
		\node at (1+0.2,2+0.2) {\tiny 2};
		\node at (1+0.2,3+0.2) {\small 4};
		\node at (1+0.2,4+0.2) {\tiny 2};
		\node at (1+0.2,5+0.2) {\small 4};
		\node at (1+0.2,6+0.2) {\tiny 2};
		\node at (1+0.2,7+0.2) {\small 4};
		\node at (1+0.2,8+0.2) {\tiny 2};
		\node at (1+0.2,9+0.2) {\small 4};
		\node at (1+0.2,10+0.2) {\tiny 1};

		\node at (2+0.2,0+0.2) {\tiny 1};
		\node at (2+0.2,1+0.2) {\tiny 1};
		\node at (2+0.2,2+0.2) {\tiny 2};
		\node at (2+0.2,3+0.2) {\tiny 1};
		\node at (2+0.2,4+0.2) {\tiny 2};
		\node at (2+0.2,5+0.2) {\tiny 1};
		\node at (2+0.2,6+0.2) {\tiny 2};
		\node at (2+0.2,7+0.2) {\tiny 1};
		\node at (2+0.2,8+0.2) {\tiny 2};
		\node at (2+0.2,9+0.2) {\tiny 1};
		\node at (2+0.2,10+0.2) {\tiny 1};

		\draw[step=1cm,thick] (0,0) -- (0,10);
		
		\draw[step=1cm,thick] (0,0) -- (2,0);
		\foreach \x in {0,1,2}
		\foreach \y in {0,1,2,3,4,5,6,7,8,9,10}
		\draw [fill] (\x,\y) circle [radius=0.05];
		
		\draw (-0.4,1) node {\small 1/2};
		\draw (-0.4,2) node {\small 1};
		\draw (-0.4,3) node {\small 3/2};
		\draw (-0.4,4) node {\small 2};
		\draw (-0.4,5) node {\small 5/2};
		\draw (-0.4,6) node {\small 3};
		\draw (-0.4,7) node {\small 7/2};
		\draw (-0.4,8) node {\small 4};
		\draw (-0.4,9) node {\small 9/2};
		\draw (-0.4,10) node {\small 5};

		\draw (1,-0.4) node {\small 1/2};
		\draw (2,-0.4) node {\small 1};
		
		\draw (-0.2,0-0.2) node {\small 0};
		
		\hspace{4cm}
		\draw[step=1cm,very thin,dashed] (0,0) grid (2,14);	
		\node at (1,15) {\textbf{Weight} [2,8]};

		\node at (0+0.2,0+0.2) {\tiny 1};
		\node at (0+0.2,1+0.2) {\tiny 1};
		\node at (0+0.2,2+0.2) {\tiny 2};
		\node at (0+0.2,3+0.2) {\tiny 1};
		\node at (0+0.2,4+0.2) {\tiny 2};
		\node at (0+0.2,5+0.2) {\tiny 1};
		\node at (0+0.2,6+0.2) {\tiny 2};
		\node at (0+0.2,7+0.2) {\tiny 1};
		\node at (0+0.2,8+0.2) {\tiny 2};
		\node at (0+0.2,9+0.2) {\tiny 1};
		\node at (0+0.2,10+0.2) {\tiny 2};
		\node at (0+0.2,11+0.2) {\tiny 1};
		\node at (0+0.2,12+0.2) {\tiny 2};
		\node at (0+0.2,13+0.2) {\tiny 1};
		\node at (0+0.2,14+0.2) {\tiny 1};

		\node at (1+0.2,0+0.2) {\tiny 1};
		\node at (1+0.2,1+0.2) {\small 4};
		\node at (1+0.2,2+0.2) {\tiny 2};
		\node at (1+0.2,3+0.2) {\small 4};
		\node at (1+0.2,4+0.2) {\tiny 2};
		\node at (1+0.2,5+0.2) {\small 4};
		\node at (1+0.2,6+0.2) {\tiny 2};
		\node at (1+0.2,7+0.2) {\small 4};
		\node at (1+0.2,8+0.2) {\tiny 2};
		\node at (1+0.2,9+0.2) {\small 4};
		\node at (1+0.2,10+0.2) {\tiny 2};
		\node at (1+0.2,11+0.2) {\small 4};
		\node at (1+0.2,12+0.2) {\tiny 2};
		\node at (1+0.2,13+0.2) {\small 4};
		\node at (1+0.2,14+0.2) {\tiny 1};

		\node at (2+0.2,0+0.2) {\tiny 1};
		\node at (2+0.2,1+0.2) {\tiny 1};
		\node at (2+0.2,2+0.2) {\tiny 2};
		\node at (2+0.2,3+0.2) {\tiny 1};
		\node at (2+0.2,4+0.2) {\tiny 2};
		\node at (2+0.2,5+0.2) {\tiny 1};
		\node at (2+0.2,6+0.2) {\tiny 2};
		\node at (2+0.2,7+0.2) {\tiny 1};
		\node at (2+0.2,8+0.2) {\tiny 2};
		\node at (2+0.2,9+0.2) {\tiny 1};
		\node at (2+0.2,10+0.2) {\tiny 2};
		\node at (2+0.2,11+0.2) {\tiny 1};
		\node at (2+0.2,12+0.2) {\tiny 2};
		\node at (2+0.2,13+0.2) {\tiny 1};
		\node at (2+0.2,14+0.2) {\tiny 1};

		\draw[step=1cm,thick] (0,0) -- (0,14);
		
		\draw[step=1cm,thick] (0,0) -- (2,0);
		\foreach \x in {0,1,2}
		\foreach \y in {0,1,2,3,4,5,6,7,8,9,10,11,12,13,14}
		\draw [fill] (\x,\y) circle [radius=0.05];
		
		\draw (-0.44,1) node {\small 1/2};
		\draw (-0.44,2) node {\small 1};
		\draw (-0.44,3) node {\small 3/2};
		\draw (-0.44,4) node {\small 2};
		\draw (-0.44,5) node {\small 5/2};
		\draw (-0.44,6) node {\small 3};
		\draw (-0.44,7) node {\small 7/2};
		\draw (-0.44,8) node {\small 4};
		\draw (-0.44,9) node {\small 9/2};
		\draw (-0.44,10) node {\small 5};
		\draw (-0.44,11) node {\small 11/2};
		\draw (-0.44,12) node {\small 6};
		\draw (-0.44,13) node {\small 13/2};
		\draw (-0.44,14) node {\small 7};

		\draw (1,-0.4) node {\small 1/2};
		\draw (2,-0.4) node {\small 1};
		
		\draw (-0.2,0-0.2) node {\small 0};

		\useasboundingbox (15,0);

		\end{tikzpicture}

		We draw similar figures for $p=2$ in level $U_0(8)$ which give:
		
		
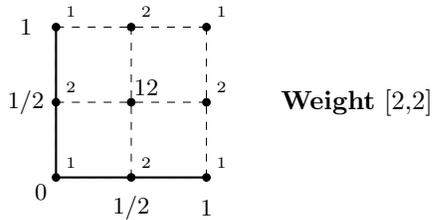
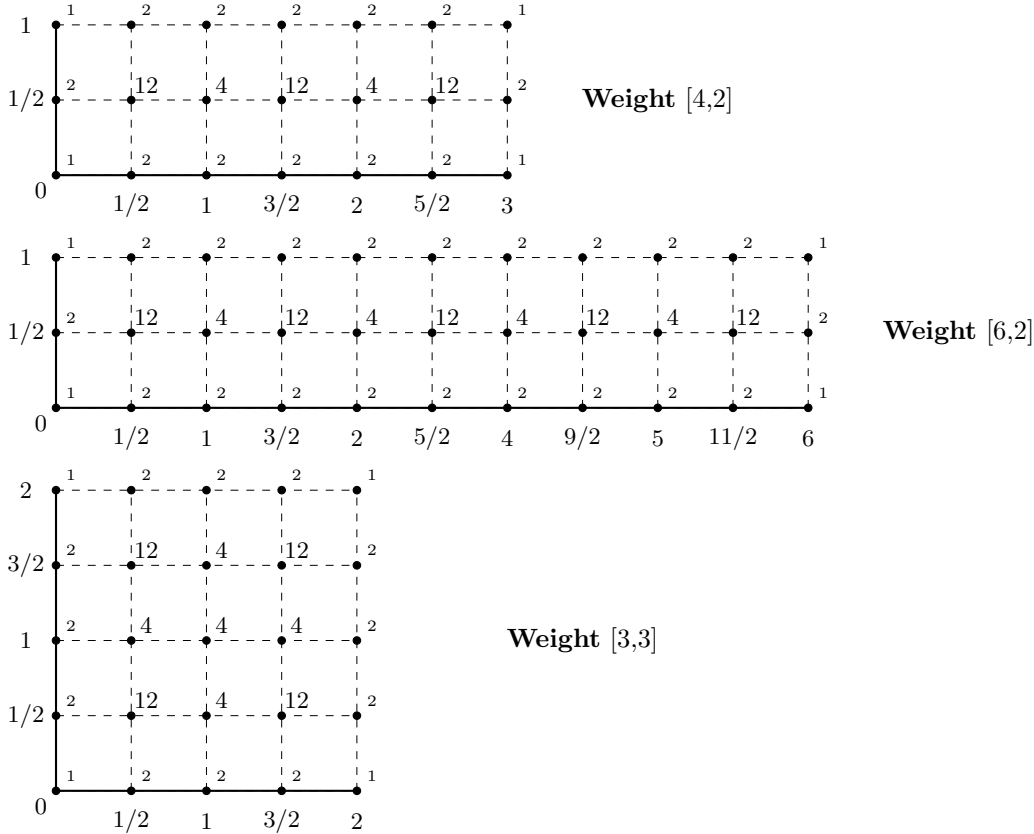
\begin{tikzpicture}

		\foreach \x in {0,1,2}
		\foreach \y in {0,1,2}
		\draw [fill] (\x,\y) circle [radius=0.05];
		
		\node at (1+0.2,1+0.2) {\small 12};
		\node at (0+0.2,0+0.2) {\tiny 1};
		\node at (1+0.2,0+0.2) {\tiny 2};
		\node at (0+0.2,1+0.2) {\tiny 2};
		\node at (1+0.2,2+0.2) {\tiny 2};
		\node at (2+0.2,1+0.2) {\tiny 2};
		\node at (2+0.2,2+0.2) {\tiny 1};
		\node at (2+0.2,0+0.2) {\tiny 1};
		\node at (0+0.2,2+0.2) {\tiny 1};
		
		\node at (4,1) {\textbf{Weight} [2,2]};
		
		\draw[step=1cm,thick] (0,0) -- (0,2);
		
		\draw[step=1cm,thick] (0,0) -- (2,0);
		\draw[step=1cm,very thin,dashed] (0,0) grid (2,2);

		\draw (1,-0.4) node {\small 1/2};
		\draw (2,-0.4) node {\small 1};
		
		\draw (-0.2,0-0.2) node {\small 0};
		\draw (-0.4,1) node {\small 1/2};
		\draw (-0.4,2) node {\small 1};

		\end{tikzpicture}
		
		\begin{tikzpicture}

		\draw[step=1cm,very thin,dashed] (0,0) grid (6,2);	
		\draw (1,-0.4) node {\small 1/2};
		\draw (2,-0.4) node {\small 1};
		\draw (3,-0.4) node {\small 3/2};
		\draw (4,-0.4) node {\small 2};
		\draw (5,-0.4) node {\small 5/2};
		\draw (6,-0.4) node {\small 3};

		\draw (-0.4,1) node {\small 1/2};
		\draw (-0.4,2) node {\small 1};
		
		\draw (-0.2,0-0.2) node {\small 0};
		
		\node at (0+0.2,0+0.2) {\tiny 1};
		\node at (1+0.2,0+0.2) {\tiny 2};
		\node at (2+0.2,0+0.2) {\tiny 2};
		\node at (3+0.2,0+0.2) {\tiny 2};
		\node at (4+0.2,0+0.2) {\tiny 2};
		\node at (5+0.2,0+0.2) {\tiny 2};
		\node at (6+0.2,0+0.2) {\tiny 1};

		\node at (0+0.2,1+0.2) {\tiny 2};
		\node at (1+0.2,1+0.2) {\small 12};
		\node at (2+0.2,1+0.2) {\small 4};
		\node at (3+0.2,1+0.2) {\small 12};
		\node at (4+0.2,1+0.2) {\small 4};
		\node at (5+0.2,1+0.2) {\small 12};
		\node at (6+0.2,1+0.2) {\tiny 2};

		\node at (0+0.2,2+0.2) {\tiny 1};
		\node at (1+0.2,2+0.2) {\tiny 2};
		\node at (2+0.2,2+0.2) {\tiny 2};
		\node at (3+0.2,2+0.2) {\tiny 2};
		\node at (4+0.2,2+0.2) {\tiny 2};
		\node at (5+0.2,2+0.2) {\tiny 2};
		\node at (6+0.2,2+0.2) {\tiny 1};

		\node at (8,1) {\textbf{Weight} [4,2]};

		\draw[step=1cm,thick] (0,0) -- (0,2);
		
		\draw[step=1cm,thick] (0,0) -- (6,0);
		\foreach \x in {0,1,2}
		\foreach \y in {0,1,2,3,4,5,6}
		\draw [fill] (\y,\x) circle [radius=0.05];

		\end{tikzpicture}
		
		\begin{tikzpicture}
		
		\draw[step=1cm,very thin,dashed] (0,0) grid (10,2);	
		\draw (1,-0.4) node {\small 1/2};
		\draw (2,-0.4) node {\small 1};
		\draw (3,-0.4) node {\small 3/2};
		\draw (4,-0.4) node {\small 2};
		\draw (5,-0.4) node {\small 5/2};
		\draw (6,-0.4) node {\small 4};
		\draw (7,-0.4) node {\small 9/2};
		\draw (8,-0.4) node {\small 5};
		\draw (9,-0.4) node {\small 11/2};
		\draw (10,-0.4) node {\small 6};

		\draw (-0.4,1) node {\small 1/2};
		\draw (-0.4,2) node {\small 1};
		
		\draw (-0.2,0-0.2) node {\small 0};
		
		\node at (0+0.2,0+0.2) {\tiny 1};
		\node at (1+0.2,0+0.2) {\tiny 2};
		\node at (2+0.2,0+0.2) {\tiny 2};
		\node at (3+0.2,0+0.2) {\tiny 2};
		\node at (4+0.2,0+0.2) {\tiny 2};
		\node at (5+0.2,0+0.2) {\tiny 2};
		\node at (6+0.2,0+0.2) {\tiny 2};
		\node at (7+0.2,0+0.2) {\tiny 2};
		\node at (8+0.2,0+0.2) {\tiny 2};
		\node at (9+0.2,0+0.2) {\tiny 2};
		\node at (10+0.2,0+0.2) {\tiny 1};

		\node at (0+0.2,1+0.2) {\tiny 2};
		\node at (1+0.2,1+0.2) {\small 12};
		\node at (2+0.2,1+0.2) {\small 4};
		\node at (3+0.2,1+0.2) {\small 12};
		\node at (4+0.2,1+0.2) {\small 4};
		\node at (5+0.2,1+0.2) {\small 12};
		\node at (6+0.2,1+0.2) {\small 4};
		\node at (7+0.2,1+0.2) {\small 12};
		\node at (8+0.2,1+0.2) {\small 4};
		\node at (9+0.2,1+0.2) {\small 12};
		\node at (10+0.2,1+0.2) {\tiny 2};

		\node at (0+0.2,2+0.2) {\tiny 1};
		\node at (1+0.2,2+0.2) {\tiny 2};
		\node at (2+0.2,2+0.2) {\tiny 2};
		\node at (3+0.2,2+0.2) {\tiny 2};
		\node at (4+0.2,2+0.2) {\tiny 2};
		\node at (5+0.2,2+0.2) {\tiny 2};
		\node at (6+0.2,2+0.2) {\tiny 2};
		\node at (7+0.2,2+0.2) {\tiny 2};
		\node at (8+0.2,2+0.2) {\tiny 2};
		\node at (9+0.2,2+0.2) {\tiny 2};
		\node at (10+0.2,2+0.2) {\tiny 1};

		\node at (12,1) {\textbf{Weight} [6,2] };

		\draw[step=1cm,thick] (0,0) -- (0,2);
		
		\draw[step=1cm,thick] (0,0) -- (10,0);
		\foreach \x in {0,1,2}
		\foreach \y in {0,1,2,3,4,5,6,7,8,9,10}
		\draw [fill] (\y,\x) circle [radius=0.05];

		\end{tikzpicture}

		\begin{tikzpicture}
		
		\draw[step=1cm,very thin,dashed] (0,0) grid (4,4);	
		\node at (7,2) {\textbf{Weight} [3,3] };

		\node at (0+0.2,0+0.2) {\tiny 1};
		\node at (0+0.2,1+0.2) {\tiny 2};
		\node at (0+0.2,2+0.2) {\tiny 2};
		\node at (0+0.2,3+0.2) {\tiny 2};
		\node at (0+0.2,4+0.2) {\tiny 1};

		\node at (1+0.2,0+0.2) {\tiny 2};
		\node at (1+0.2,1+0.2) {\small 12};
		\node at (1+0.2,2+0.2) {\small 4};
		\node at (1+0.2,3+0.2) {\small 12};
		\node at (1+0.2,4+0.2) {\tiny 2};

		\node at (2+0.2,0+0.2) {\tiny 2};
		\node at (2+0.2,1+0.2) {\small 4};
		\node at (2+0.2,2+0.2) {\small 4};
		\node at (2+0.2,3+0.2) {\small 4};
		\node at (2+0.2,4+0.2) {\tiny 2};

		\node at (3+0.2,0+0.2) {\tiny 2};
		\node at (3+0.2,1+0.2) {\small 12};
		\node at (3+0.2,2+0.2) {\small 4};
		\node at (3+0.2,3+0.2) {\small 12};
		\node at (3+0.2,4+0.2) {\tiny 2};

		\node at (4+0.2,0+0.2) {\tiny 1};
		\node at (4+0.2,1+0.2) {\tiny 2};
		\node at (4+0.2,2+0.2) {\tiny 2};
		\node at (4+0.2,3+0.2) {\tiny 2};
		\node at (4+0.2,4+0.2) {\tiny 1};

		\draw[step=1cm,thick] (0,0) -- (0,4);
		
		\draw[step=1cm,thick] (0,0) -- (4,0);
		\foreach \x in {0,1,2,3,4}
		\foreach \y in {0,1,2,3,4}
		\draw [fill] (\x,\y) circle [radius=0.05];
		
		\draw (-0.4,1) node {\small 1/2};
		\draw (-0.4,2) node {\small 1};
		\draw (-0.4,3) node {\small 3/2};
		\draw (-0.4,4) node {\small 2};

		\draw (1,-0.4) node {\small 1/2};
		\draw (2,-0.4) node {\small 1};
		\draw (3,-0.4) node {\small 3/2};
		\draw (4,-0.4) node {\small 2};

		\draw (-0.2,0-0.2) node {\small 0};
		\end{tikzpicture}
		
	\added[id=c]{ Very recent work of Newton and Johansson sheds some light onto why, the slopes for $U_{\ps}$ have such behaviour as observed in \ref{obsplit}(4).}

		\begin{num}
			\added[id=h]{Since we are in the classical case there is no problem in computing the slopes of $U_{\ps_1},U_{\ps_2}$, from which we can construct the above figures as follows: thinking of the multiplicities as variables $(x_{i,j})$, the slopes of $U_{\ps_1}^aU_{\ps_2}^b$ for varying $a,b$ give us linear equations in $(x_{i,j})$ which one can try to solve. For example, knowing that in weight $[2,2]\psi_2$ and level $U_0(9)$ the operator $U_{\ps_1}$ has slopes $[(  0 , 3 ) , (  1/2 , 6 ) ,  (  1 , 3 )]$	tells us that in the corresponding figure adding the multiplicities along each column should give $3,6,3$ respectively. Furthermore, the Atkin-Lehner involutions $W_p,W_{\ps_i}$ give extra symmetries in the multiplicities, e.g., $W_p$ sends the pair $$(\l_{\ps_1},\l_{\ps_{2}}) \longmapsto \left (k_0-1-\l_{\ps_1}-v_{\gothp_1}(k),k_0-1-\l_{\ps_{2}}-v_{\ps_2}(k)\right )$$ which combined give us enough equations to uniquely determine the multiplicities (for the weights in the above figure).}

		\end{num}

		\begin{obs}
			We note that in both examples above, the pictures appear to built up from the weight $[2,2]$ picture, by `glueing' along the edges and adding up the multiplicities along the edges. 
		\end{obs}

		\begin{question} For (arithmetic) weights near the boundary, are the above multiplicities all ways greater than $0$? In other words, given eigenforms $f_i$ for $U_{\ps_i}$ with eigenvalues $\a_i$, does there exist an eigenform for $U_p$ with eigenvalue $\a_1 \a_2 \cdots \a_f$.
		\end{question}
		
		\begin{question} Can we obtain the picture above for any weight near the boundary, by simply glueing the picture in weight $[2,2]$? 
		\end{question}

		\subsubsection{\textbf{ Overconvergent partial slopes}}\label{ops}

		Now, since we are in the overconvergent case one cannot directly compute the slopes of $U_{\ps_i}$ since these are not compact operators. Instead one can compute  the successive slopes of $U_p U_{\ps_i}^n$ (which are compact operators) for $n \ge 0$. From this one can obtain slopes of $U_{\ps_i}$ as follows: let $R \gg 0$ , $P_n(R,\k):=(U_p U_{\ps_i}^n)(R,\k)$ (with the same notation as in \ref{wow}). Let us for the moment consider $\S(P_{0}(R,\k))$ as a multiset of slopes in the obvious way.  Now, for each $s \in \S(P_{0}(R,\k))$, let \[T(s) =\bigcap_{n=1}^{J(s)} \left \{ (t-s)/n \mid t \in \S(P_n(R,\k))   \right \}\] where $J(s) \gg 0$ such that the intersection stabilizes (such a $J(s)$ always exists). Then (for large enough $R$) \[\underset{s \in \S(P_{0}(R,\k))}{\bigcup}T(s) \subset \S(U_{\ps_i})\] which is what we want.

		Alternatively, while $U_{\ps_i}$ are not compact operators on the spaces of overconvergent Hilbert modular forms, one can restrict them to subspaces on which they act as compact operators. To see this, let $L(n,m)$ denote the subspace of $L \langle X,Y \rangle$ generated by $X^i Y^j$ for $i \in \{0,\dots,n\}$ and $j \in \{0,\dots,m\}$ where $n,m \in \ZZ_{\geq 0} \cup \{\infty\}$ ( note that $L(\infty,\infty)=L \langle X,Y \rangle$). Then for $m \in \ZZ_{\geq 0}$,  \added[id=h]{$\k=[\k_1,\k_2]\chi$} a weight with $k_2=m+2$ and $k_1$ arbitrary (with the same parity as $k_2$), the subspace \[\bigoplus_{i=1}^h L(\infty,m) \subset \bigoplus_{i=1}^h L \langle X,Y \rangle \cong S_{\k}^{D,\dag}(U)\] is  for a fixed under the $|_\k$ action of Hecke operators and $U_{\ps_1}$ acts compactly (this can be seen from  Corollary \ref{cor714}) on this subspace. Similarly $U_{\ps_2}$ is compact on $\bigoplus_{i=1}^h L(n,\infty)$ for a fixed $n \in \ZZ_{\geq 0}$ and weights \added[id=h]{$\k=[n+2,k_2]\chi$}. From this one can compute subsets of $\S(U_{\ps_i})$.

		Using this we compute some overconvergent slopes of $U_{\ps_i}$ in weight $[2,2]\psi_2$ and level $U_0(9)$ acting on $\oplus_i L(0,8) \subset \oplus_i L(0,\infty)$ and $\oplus_i L(8,0) \subset \oplus_i L(\infty,0)$. We only show this for $U_{\ps_1}$ on $\oplus_i L(8,0)$ since the picture for $U_{\ps_2}$ on $\oplus_i L(8,0)$ is the same but flipped vertically. Note that $\oplus_i L(0,0) \cong S_{2}(U)$.
		
		\resizebox{\textwidth}{!}
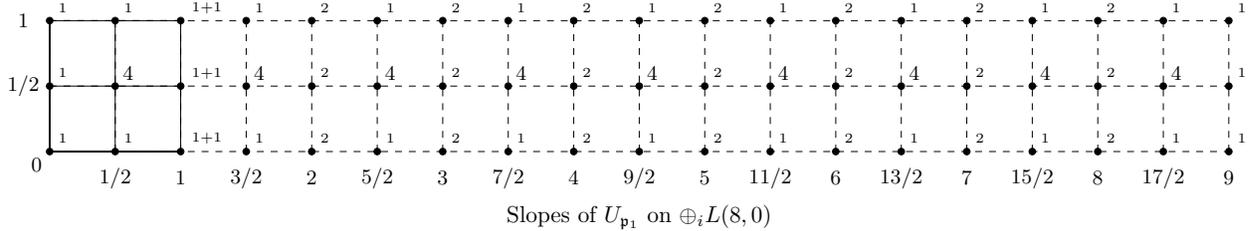
{
			\hspace{-0.5cm}
			\begin{tikzpicture}
			\draw[step=1cm,very thin,dashed] (0,0) grid (18,2);	
			\draw[step=1cm,very thin] (0,0) grid (2,2);	
			\draw (1,-0.4) node {\small 1/2};
			\draw (2,-0.4) node {\small 1};
			\draw (3,-0.4) node {\small 3/2};
			\draw (4,-0.4) node {\small 2};
			\draw (5,-0.4) node {\small 5/2};
			\draw (6,-0.4) node {\small 3};
			\draw (7,-0.4) node {\small 7/2};
			\draw (8,-0.4) node {\small 4};
			\draw (9,-0.4) node {\small 9/2};
			\draw (10,-0.4) node {\small 5};
			\draw (11,-0.4) node {\small 11/2};
			\draw (12,-0.4) node {\small 6};
			\draw (13,-0.4) node {\small 13/2};
			\draw (14,-0.4) node {\small 7};
			\draw (15,-0.4) node {\small 15/2};
			\draw (16,-0.4) node {\small 8};
			\draw (17,-0.4) node {\small 17/2};
			\draw (18,-0.4) node {\small 9};

			\draw (-0.4,1) node {\small 1/2};
			\draw (-0.4,2) node {\small 1};
			
			\draw (-0.2,0-0.2) node {\small 0};
			
			\node at (0+0.2,0+0.2) {\tiny 1};
			\node at (1+0.2,0+0.2) {\tiny 1};
			\node at (2+0.4,0+0.2) {\tiny 1+1};
			\node at (3+0.2,0+0.2) {\tiny 1};
			\node at (4+0.2,0+0.2) {\tiny 2};
			\node at (5+0.2,0+0.2) {\tiny 1};
			\node at (6+0.2,0+0.2) {\tiny 2};
			\node at (7+0.2,0+0.2) {\tiny 1};
			\node at (8+0.2,0+0.2) {\tiny 2};
			\node at (9+0.2,0+0.2) {\tiny 1};
			\node at (10+0.2,0+0.2) {\tiny 2};
			\node at (11+0.2,0+0.2) {\tiny 1};
			\node at (12+0.2,0+0.2) {\tiny 2};
			\node at (13+0.2,0+0.2) {\tiny 1};
			\node at (14+0.2,0+0.2) {\tiny 2};
			\node at (15+0.2,0+0.2) {\tiny 1};
			\node at (16+0.2,0+0.2) {\tiny 2};
			\node at (17+0.2,0+0.2) {\tiny 1};
			\node at (18+0.2,0+0.2) {\tiny 1};

			\node at (0+0.2,1+0.2) {\tiny 1};
			\node at (1+0.2,1+0.2) {\small 4};
			\node at (2+0.4,1+0.2) {\tiny 1+1};
			\node at (3+0.2,1+0.2) {\small 4};
			\node at (4+0.2,1+0.2) {\tiny 2};
			\node at (5+0.2,1+0.2) {\small 4};
			\node at (6+0.2,1+0.2) {\tiny 2};
			\node at (7+0.2,1+0.2) {\small 4};
			\node at (8+0.2,1+0.2) {\tiny 2};
			\node at (9+0.2,1+0.2) {\small 4};
			\node at (10+0.2,1+0.2) {\tiny 2};
			\node at (11+0.2,1+0.2) {\small 4};
			\node at (12+0.2,1+0.2) {\tiny 2};
			\node at (13+0.2,1+0.2) {\small 4};
			\node at (14+0.2,1+0.2) {\tiny 2};
			\node at (15+0.2,1+0.2) {\small 4};
			\node at (16+0.2,1+0.2) {\tiny 2};
			\node at (17+0.2,1+0.2) {\small 4};
			\node at (18+0.2,1+0.2) {\tiny 1};

			\node at (0+0.2,2+0.2) {\tiny 1};
			\node at (1+0.2,2+0.2) {\tiny 1};
			\node at (2+0.4,2+0.2) {\tiny 1+1};
			\node at (3+0.2,2+0.2) {\tiny 1};
			\node at (4+0.2,2+0.2) {\tiny 2};
			\node at (5+0.2,2+0.2) {\tiny 1};
			\node at (6+0.2,2+0.2) {\tiny 2};
			\node at (7+0.2,2+0.2) {\tiny 1};
			\node at (8+0.2,2+0.2) {\tiny 2};
			\node at (9+0.2,2+0.2) {\tiny 1};
			\node at (10+0.2,2+0.2) {\tiny 2};
			\node at (11+0.2,2+0.2) {\tiny 1};
			\node at (12+0.2,2+0.2) {\tiny 2};
			\node at (13+0.2,2+0.2) {\tiny 1};
			\node at (14+0.2,2+0.2) {\tiny 2};
			\node at (15+0.2,2+0.2) {\tiny 1};
			\node at (16+0.2,2+0.2) {\tiny 2};
			\node at (17+0.2,2+0.2) {\tiny 1};
			\node at (18+0.2,2+0.2) {\tiny 1};

			\node at (9,-1) {Slopes of $U_{\ps_1}$ on $\oplus_i L(8,0)$};

			\draw[step=1cm,thick] (0,0) -- (0,2);
			
			\draw[step=1cm,thick] (0,0) -- (2,0);
			\foreach \x in {0,1,2}
			\foreach \y in {0,1,2,3,4,5,6,7,8,9,10,11,12,13,14,15,16,17,18}
			\draw [fill] (\y,\x) circle [radius=0.05];

			\end{tikzpicture}
		}

		\resizebox{\textwidth}{!}
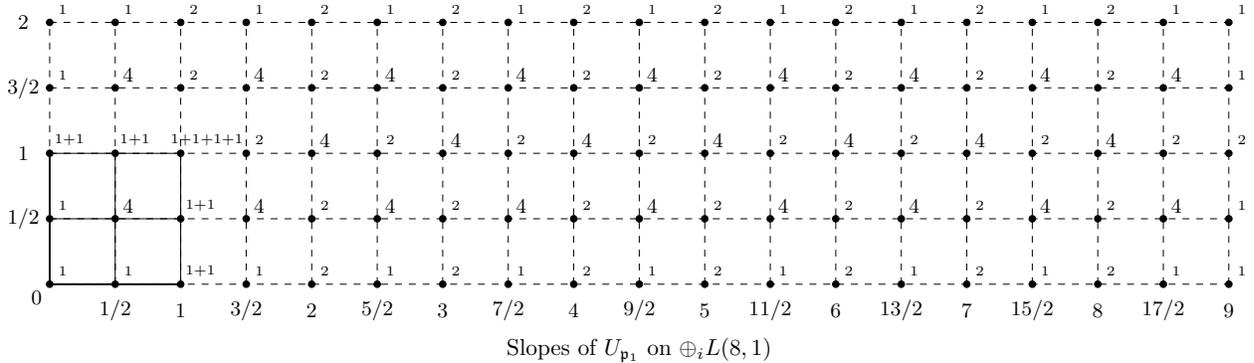
{
			\hspace{-0.5 cm}
			\begin{tikzpicture}
			
			\draw[step=1cm,very thin,dashed] (0,0) grid (18,4);	
			\draw[step=1cm,very thin] (0,0) grid (2,2);	
			\draw (1,-0.4) node {\small 1/2};
			\draw (2,-0.4) node {\small 1};
			\draw (3,-0.4) node {\small 3/2};
			\draw (4,-0.4) node {\small 2};
			\draw (5,-0.4) node {\small 5/2};
			\draw (6,-0.4) node {\small 3};
			\draw (7,-0.4) node {\small 7/2};
			\draw (8,-0.4) node {\small 4};
			\draw (9,-0.4) node {\small 9/2};
			\draw (10,-0.4) node {\small 5};
			\draw (11,-0.4) node {\small 11/2};
			\draw (12,-0.4) node {\small 6};
			\draw (13,-0.4) node {\small 13/2};
			\draw (14,-0.4) node {\small 7};
			\draw (15,-0.4) node {\small 15/2};
			\draw (16,-0.4) node {\small 8};
			\draw (17,-0.4) node {\small 17/2};
			\draw (18,-0.4) node {\small 9};

			\draw (-0.4,3) node {\small 3/2};
			\draw (-0.4,4) node {\small 2};
			\draw (-0.4,1) node {\small 1/2};
			\draw (-0.4,2) node {\small 1};
			
			\draw (-0.2,0-0.2) node {\small 0};
			
			\node at (0+0.2,0+0.2) {\tiny 1};
			\node at (1+0.2,0+0.2) {\tiny 1};
			\node at (2+0.3,0+0.2) {\tiny 1+1};
			\node at (3+0.2,0+0.2) {\tiny 1};
			\node at (4+0.2,0+0.2) {\tiny 2};
			\node at (5+0.2,0+0.2) {\tiny 1};
			\node at (6+0.2,0+0.2) {\tiny 2};
			\node at (7+0.2,0+0.2) {\tiny 1};
			\node at (8+0.2,0+0.2) {\tiny 2};
			\node at (9+0.2,0+0.2) {\tiny 1};
			\node at (10+0.2,0+0.2) {\tiny 2};
			\node at (11+0.2,0+0.2) {\tiny 1};
			\node at (12+0.2,0+0.2) {\tiny 2};
			\node at (13+0.2,0+0.2) {\tiny 1};
			\node at (14+0.2,0+0.2) {\tiny 2};
			\node at (15+0.2,0+0.2) {\tiny 1};
			\node at (16+0.2,0+0.2) {\tiny 2};
			\node at (17+0.2,0+0.2) {\tiny 1};
			\node at (18+0.2,0+0.2) {\tiny 1};
			
			\node at (0+0.2,4+0.2) {\tiny 1};
			\node at (1+0.2,4+0.2) {\tiny 1};
			\node at (2+0.2,4+0.2) {\tiny 2};
			\node at (3+0.2,4+0.2) {\tiny 1};
			\node at (4+0.2,4+0.2) {\tiny 2};
			\node at (5+0.2,4+0.2) {\tiny 1};
			\node at (6+0.2,4+0.2) {\tiny 2};
			\node at (7+0.2,4+0.2) {\tiny 1};
			\node at (8+0.2,4+0.2) {\tiny 2};
			\node at (9+0.2,4+0.2) {\tiny 1};
			\node at (10+0.2,4+0.2) {\tiny 2};
			\node at (11+0.2,4+0.2) {\tiny 1};
			\node at (12+0.2,4+0.2) {\tiny 2};
			\node at (13+0.2,4+0.2) {\tiny 1};
			\node at (14+0.2,4+0.2) {\tiny 2};
			\node at (15+0.2,4+0.2) {\tiny 1};
			\node at (16+0.2,4+0.2) {\tiny 2};
			\node at (17+0.2,4+0.2) {\tiny 1};
			\node at (18+0.2,4+0.2) {\tiny 1};

			\node at (0+0.2,1+0.2) {\tiny 1};
			\node at (1+0.2,1+0.2) {\small 4};
			\node at (2+0.3,1+0.2) {\tiny 1+1};
			\node at (3+0.2,1+0.2) {\small 4};
			\node at (4+0.2,1+0.2) {\tiny 2};
			\node at (5+0.2,1+0.2) {\small 4};
			\node at (6+0.2,1+0.2) {\tiny 2};
			\node at (7+0.2,1+0.2) {\small 4};
			\node at (8+0.2,1+0.2) {\tiny 2};
			\node at (9+0.2,1+0.2) {\small 4};
			\node at (10+0.2,1+0.2) {\tiny 2};
			\node at (11+0.2,1+0.2) {\small 4};
			\node at (12+0.2,1+0.2) {\tiny 2};
			\node at (13+0.2,1+0.2) {\small 4};
			\node at (14+0.2,1+0.2) {\tiny 2};
			\node at (15+0.2,1+0.2) {\small 4};
			\node at (16+0.2,1+0.2) {\tiny 2};
			\node at (17+0.2,1+0.2) {\small 4};
			\node at (18+0.2,1+0.2) {\tiny 1};

			\node at (0+0.3,2+0.2) {\tiny 1+1};
			\node at (1+0.3,2+0.2) {\tiny 1+1};
			\node at (2+0.4,2+0.2) {\tiny 1+1+1+1};
			\node at (3+0.2,2+0.2) {\tiny 2};
			\node at (4+0.2,2+0.2) {\small 4};
			\node at (5+0.2,2+0.2) {\tiny 2};
			\node at (6+0.2,2+0.2)  {\small 4};
			\node at (7+0.2,2+0.2) {\tiny 2};
			\node at (8+0.2,2+0.2)  {\small 4};
			\node at (9+0.2,2+0.2) {\tiny 2};
			\node at (10+0.2,2+0.2)  {\small 4};
			\node at (11+0.2,2+0.2) {\tiny 2};
			\node at (12+0.2,2+0.2)  {\small 4};
			\node at (13+0.2,2+0.2) {\tiny 2};
			\node at (14+0.2,2+0.2)  {\small 4};
			\node at (15+0.2,2+0.2) {\tiny 2};
			\node at (16+0.2,2+0.2)  {\small 4};
			\node at (17+0.2,2+0.2) {\tiny 2};
			\node at (18+0.2,2+0.2) {\tiny 2};
			
			\node at (0+0.2,3+0.2) {\tiny 1};
			\node at (1+0.2,3+0.2) {\small 4};
			\node at (2+0.2,3+0.2) {\tiny 2};
			\node at (3+0.2,3+0.2) {\small 4};
			\node at (4+0.2,3+0.2) {\tiny 2};
			\node at (5+0.2,3+0.2) {\small 4};
			\node at (6+0.2,3+0.2) {\tiny 2};
			\node at (7+0.2,3+0.2) {\small 4};
			\node at (8+0.2,3+0.2) {\tiny 2};
			\node at (9+0.2,3+0.2) {\small 4};
			\node at (10+0.2,3+0.2) {\tiny 2};
			\node at (11+0.2,3+0.2) {\small 4};
			\node at (12+0.2,3+0.2) {\tiny 2};
			\node at (13+0.2,3+0.2) {\small 4};
			\node at (14+0.2,3+0.2) {\tiny 2};
			\node at (15+0.2,3+0.2) {\small 4};
			\node at (16+0.2,3+0.2) {\tiny 2};
			\node at (17+0.2,3+0.2){\small 4};
			\node at (18+0.2,3+0.2) {\tiny 1};

			\node at (9,-1) {Slopes of $U_{\ps_1}$ on $\oplus_i L(8,1)$ };

			\draw[step=1cm,thick] (0,0) -- (0,2);
			
			\draw[step=1cm,thick] (0,0) -- (2,0);
			\foreach \x in {0,1,2,3,4}
			\foreach \y in {0,1,2,3,4,5,6,7,8,9,10,11,12,13,14,15,16,17,18}
			\draw [fill] (\y,\x) circle [radius=0.05];
			
			\end{tikzpicture}
		}
		In the above figures, the classical slopes are shown as a grid with solid lines and the overconvergent slopes are dashed.
		
		\begin{subsection}{\textbf{ Inert case}}	We now move to the inert case. For this we set $F=\QQ(\sqrt{5})$ and $p=2$. We will compute the slopes of $U_2$ acting on Hilbert modular forms of level $\K_0(2^3\gothp_{11})$ where $\gothp_{11}$ is the prime lying above $11$ generated by $(11,3+2\sqrt{5})$ and weight  $[k_1,k_2]\psi_r$ where $\psi_r: \OO_p^{\times} \to \{\pm 1\}$ is the primitive Hecke character of conductor $2^3$. In particular, it is such that $\psi_r(e)=\Norm(e)^r$ where $e \in \OO_F^\times$ is the fundamental unit (embedded in the usual way into $\OO_p^\times$) is the fundamental unit.  Lastly, Let $\tau$ be as in \ref{tei}. Note that in this case $h=16$ and therefore the space of classical Hilbert modular forms of weight $[k_1,k_2]\psi_r$  and level $\K_0(2^3\gothp_{11})$ (which can be checked to be sufficiently small) has dimension $(k_1-1)\cdot (k_2-1)16$ for $[k_1,k_2] \neq [2,2]$. In weight $[2,2]\psi_2$ the dimension is $15$, but with our convention we include the reduced norm forms, so we get a 16-dimensional space. 
			
			\begin{num}
				Note that in this case $L=\QQ_2(\sqrt{5})$ is the degree $2$ unramified extension of $\QQ_2$. One then checks that the torsion subgroup of the units is cyclic of order $6$ given by the $6$-th roots of unity. Therefore, a locally algebraic weight $[k_1,k_2]\psi_r$  induces a map on the $6$-th roots of unity, and this map determines in what component of  weight space the weight lives. Now looking at the explicit description of the weight, we see that $$[k_1,k_2]\psi_r : \zeta_6 \mapsto \zeta_6^{k_1-k_2}\psi_r(\zeta_6).$$ From which it follows that for a fixed character $\psi_r$, the locally algebraic weights given by $[k_1,k_2]\psi_r$ and $[k_1',k_2']\psi_r$ will live on the same component of  weight space if and only if $k_1-k_2 \equiv k_1'-k_2' \mod 6$. Moreover, we can switch between the different components of weight space by twisting by $\tau$. 
			\end{num}
			
			\begin{longtabu}{|c |X|}
				\Xhline{2 pt} 	
				Weight  & Classical Slopes \\  \Xhline{2 pt}
				
				$[2,2]\psi_2$ &   
				
				$(2/3, 6), (1, 4), (4/3, 6)$
				\\ \hline

				$[2,2]\psi_2 \tau^2$ &   
				
				$	(1/2, 4), (1, 8), (3/2, 4)$
				\\ \hline

				$[2,4]\psi_2$ &   
				
				$(1/2, 4),
				(1, 8),
				(3/2, 4),
				(5/3, 6),
				(2, 4),
				(7/3, 6),
				(5/2, 4),
				(3, 8),
				(7/2, 4)$
				\\ \hline 	
				
				$[2,6]\psi_2$ &   
				
				$ (  1/2 , 4 ) $,
				$ (  1 , 8 ) $,
				$ (  3/2 , 8 ) $,
				$ (  2 , 8 ) $,
				$ (  5/2 , 4 ) $,
				$ (  8/3 , 6 ) $,
				$ (  3 , 4 ) $,
				$ (  10/3 , 6 ) $,
				$ (  7/2 , 4 ) $,
				$ (  4 , 8 ) $,
				$ (  9/2 , 8 ) $,
				$ (  5 , 8 ) $,
				$ (  11/2 , 4 ) $
				\\ \hline

				$[2,8]\psi_2$ &   
				$ (  2/3 , 6 ) $,
				$ (  1 , 4 ) $,
				$ (  4/3 , 6 ) $,
				$ (  3/2 , 4 ) $,
				$ (  2 , 8 ) $,
				$ (  5/2 , 8 ) $,
				$ (  3 , 8 ) $,
				$ (  7/2 , 4 ) $,
				$ (  11/3 , 6 ) $,
				$ (  4 , 4 ) $,
				$ (  13/3 , 6 ) $,
				$ (  9/2 , 4 ) $,
				$ (  5 , 8 ) $,
				$ (  11/2 , 8 ) $,
				$ (  6 , 8 ) $,
				$ (  13/2 , 4 ) $,
				$ (  20/3 , 6 ) $,
				$ (  7 , 4 ) $,
				$ (  22/3 , 6 ) $
				\\ \hline

				$[4,4]\psi_2$ &   
				
				$ (  2/3 , 6 ) $,
				$ (  1 , 4 ) $,
				$ (  4/3 , 6 ) $,
				$ (  3/2 , 8 ) $,
				$ (  2 , 16 ) $,
				$ (  5/2 , 16 ) $,
				$ (  8/3 , 6 ) $,
				$ (  3 , 20 ) $,
				$ (  10/3 , 6 ) $,
				$ (  7/2 , 16 ) $,
				$ (  4 , 16 ) $,
				$ (  9/2 , 8 ) $,
				$ (  14/3 , 6 ) $,
				$ (  5 , 4 ) $,
				$ (  16/3 , 6 ) $
				\\ \hline

				$[3,3]\psi_1$ &   
				$ (  2/3 , 6 ) $,
				$ (  1 , 4 ) $,
				$ (  4/3 , 6 ) $,
				$ (  3/2 , 8 ) $,
				$ (  2 , 16 ) $,
				$ (  5/2 , 8 ) $,
				$ (  8/3 , 6 ) $,
				$ (  3 , 4 ) $,
				$ (  10/3 , 6 ) $
				\\ \hline

				$[3,5]\psi_2$ &   
				$ (  1/2 , 4 ) $,
				$ (  1 , 8 ) $,
				$ (  3/2 , 8 ) $,
				$ (  5/3 , 6 ) $,
				$ (  2 , 12 ) $,
				$ (  7/3 , 6 ) $,
				$ (  5/2 , 12 ) $,
				$ (  3 , 16 ) $,
				$ (  7/2 , 12 ) $,
				$ (  11/3 , 6 ) $,
				$ (  4 , 12 ) $,
				$ (  13/3 , 6 ) $,
				$ (  9/2 , 8 ) $,
				$ (  5 , 8 ) $,
				$ (  11/2 , 4 ) $
				\\ \Xhline{2 pt} 
				\caption{Inert case, $p=2$, classical slopes.}	 		   		
			\end{longtabu}
			
			We now compute some overconvergent slopes, extending our previous computations. As in the split case, the computations suggest that as long as the weights are in the same component of  weight space, they have the same multiset of slopes. In the table below we let component $1$ consists  of the weights $[k_1,k_2] \psi_r$ appearing the table of classical slopes for which $k_1 \equiv k_2 \mod 6$, and component $2$ consist of the remaining weights.
			
			\begin{longtabu}{|c| c| X| }
				\Xhline{2 pt}
				Component & Matrix & Overconvergent Slopes \\ \Xhline{2 pt}				
				1 & $U_p(20 \cdot 16)$ &   		
				$ (  2/3 , 6 ) $,
				$ (  1 , 4 ) $,
				$ (  4/3 , 6 ) $,
				$ (  3/2 , 8 ) $,
				$ (  2 , 16 ) $,
				$ (  5/2 , 16 ) $,
				$ (  8/3 , 6 ) $,
				$ (  3 , 20 ) $,
				$ (  10/3 , 6 ) $,
				$ (  7/2 , 16 ) $,
				$ (  11/3 , 12 ) $,
				$ (  4 , 24 ) $,
				$ (  13/3 , 12 ) $,
				$ (  9/2 , 24 ) $,
				$ (  14/3 , 6 ) $,
				$ (  5 , 36 ) $,
				$ (  16/3 , 6 ) $,
				$ (  11/2 , 28 ) $,
				$ (  17/3 , 12 ) $,
				$ (  6 , 32 ) $,
				$ (  19/3 , 12 ) $,
				$ (  13/2 , 12 ) $	
				\\ \hline 
				1 & $U_p(30 \cdot 16)$ &   		
				
				$ (  2/3 , 6 ) $,
				$ (  1 , 4 ) $,
				$ (  4/3 , 6 ) $,
				$ (  3/2 , 8 ) $,
				$ (  2 , 16 ) $,
				$ (  5/2 , 16 ) $,
				$ (  8/3 , 6 ) $,
				$ (  3 , 20 ) $,
				$ (  10/3 , 6 ) $,
				$ (  7/2 , 16 ) $,
				$ (  11/3 , 12 ) $,
				$ (  4 , 24 ) $,
				$ (  13/3 , 12 ) $,
				$ (  9/2 , 24 ) $,
				$ (  14/3 , 6 ) $,
				$ (  5 , 36 ) $,
				$ (  16/3 , 6 ) $,
				$ (  11/2 , 32 ) $,
				$ (  17/3 , 12 ) $,
				$ (  6 , 40 ) $,
				$ (  19/3 , 12 ) $,
				$ (  13/2 , 32 ) $,
				$ (  20/3 , 18 ) $,
				$ (  7 , 44 ) $,
				$ (  22/3 , 18 ) $,
				$ (  15/2 , 24 ) $,
				$ (  8 , 16 ) $,
				$ (  17/2 , 8 ) $	
				\\ \hline 			
				2 & $U_p(20 \cdot 16)$  &   					
				$ (  1/2 , 4 ) $,
				$ (  1 , 8 ) $,
				$ (  3/2 , 8 ) $,
				$ (  5/3 , 6 ) $,
				$ (  2 , 12 ) $,
				$ (  7/3 , 6 ) $,
				$ (  5/2 , 12 ) $,
				$ (  8/3 , 6 ) $,
				$ (  3 , 20 ) $,
				$ (  10/3 , 6 ) $,
				$ (  7/2 , 20 ) $,
				$ (  11/3 , 6 ) $,
				$ (  4 , 28 ) $,
				$ (  13/3 , 6 ) $,
				$ (  9/2 , 24 ) $,
				$ (  14/3 , 12 ) $,
				$ (  5 , 32 ) $,
				$ (  16/3 , 12 ) $,
				$ (  11/2 , 24 ) $,
				$ (  17/3 , 12 ) $,
				$ (  6 , 32 ) $,
				$ (  19/3 , 12 ) $,
				$ (  13/2 , 12 ) $
				\\ \hline 								
				2 & $U_p(30 \cdot 16)$  &   					
				$ (  1/2 , 4 ) $,
				$ (  1 , 8 ) $,
				$ (  3/2 , 8 ) $,
				$ (  5/3 , 6 ) $,
				$ (  2 , 12 ) $,
				$ (  7/3 , 6 ) $,
				$ (  5/2 , 12 ) $,
				$ (  8/3 , 6 ) $,
				$ (  3 , 20 ) $,
				$ (  10/3 , 6 ) $,
				$ (  7/2 , 20 ) $,
				$ (  11/3 , 6 ) $,
				$ (  4 , 28 ) $,
				$ (  13/3 , 6 ) $,
				$ (  9/2 , 24 ) $,
				$ (  14/3 , 12 ) $,
				$ (  5 , 32 ) $,
				$ (  16/3 , 12 ) $,
				$ (  11/2 , 28 ) $,
				$ (  17/3 , 12 ) $,
				$ (  6 , 40 ) $,
				$ (  19/3 , 12 ) $,
				$ (  13/2 , 36 ) $,
				$ (  20/3 , 12 ) $,
				$ (  7 , 48 ) $,
				$ (  22/3 , 12 ) $,
				$ (  15/2 , 24 ) $,
				$ (  23/3 , 6 ) $,
				$ (  8 , 12 ) $,
				$ (  25/3 , 6 ) $,
				$ (  17/2 , 4 ) $					
				\\ \Xhline{2 pt}
				\caption{Inert case, $p=2$, overconvergent slopes.}	 			
			\end{longtabu}
			
			\begin{obs}\label{rmk74}
				\begin{enumerate}[$(1)$]	
					\item 	Here we see that the slopes are {\it not} appearing as a finite union of arithmetic progressions, in contrast to what one sees over $\QQ$, but we will see later that the slopes have a similar overall structure. Also we have this phenomenon of increasing multiplicities as in the split case. 
					
					\item In the above example, the different components of  weight space are identified by the Galois orbits of the characters. Note that we can move between the components by twisting by the character $\tau$ as we did in the weight $[2,2]$ case.
					
					\item Again one can see the symmetries in the slopes arising from the Atkin-Lehner involution.
				\end{enumerate}
			\end{obs}

		\end{subsection}

		\subsection{\textbf{ Conjectural behaviour near the boundary}}
		Over $\QQ$, \cite{bergp} have given a conjectural recipe to generate all of the overconvergent slopes and if one looks at this recipe one sees that its only `ingredients' are classical slopes appearing in weight $2$ (with appropriate character) at each component of  weight space and the number of cusps. The analogous behaviour is present in our computations and in general, our computations suggest the following:
		
		\begin{nota}
		\added[id=c]{	Recall that we call a subset $B$ of $\QQ_{\geq 0} \times \ZZ_{\geq 1}$ a set $sm$-pairs. For $(s,m) \in B$, we call $s$ the slope and $m$ the multiplicity. For $B$ any set of $sm$-pairs, we let $\mathfrak S(B)=|\{\sum b : (a,b) \in B \}|$, in other words, $\mathfrak{S}(B)$ is the sum of the multiplicities of $B$.}
		
		\added[id=c]{Similarly, for a  $V \in \{U_p,U_{\ps}\}$ we defined $\S_{\k}(V)$ to be the set of $sm$-pairs $(s,m)$ where $s$ is the slope of an eigenvalue of $V$ (acting in weight $\k$) and $m$ is the multiplicity with which it appears. Lastly, let $\S_{\k}^{cl}(V)=\S_{\k} (V^{cl})$ where $V^{cl}$ denotes the restriction of $V$ to the classical subspace of Hilbert modular forms.} 
			
		\end{nota}
		
		\begin{con}\label{splitgencon}
			Let $U$ be a sufficiently small level and let $\k=[k_1,k_2]\psi$ be a locally algebraic weight near the boundary. Let $V \in \{U_p,U_{\gothp}\}$, \added[id=c]{then for each $t \in (\ZZ/T\ZZ)^2$ (with $T$ the order of\/ $\OO_{p,tors}^\times$) there exists a set $B_\k(t,V)$ of $sm$-pairs with $\mathfrak S(B_\k(t,V))=h$  (the class number of $(D,U)$)}c  which only depends on which component $\k$ lies in, such that  (after scaling the slopes by $\val_p(w(\k))$)
			
			$$\S_\k(V)=\bigcup_{t \in \ZZ_{\geq 0}^2} \left \{B_\k(\ov{t},V) + l(t) \right \}$$ where\added[id=h]{ $\ov{t}$ is the image of $t$ in $(\ZZ/T\ZZ)^2$ and $l(t)=t_1+t_2$ for $t=(t_1,t_2)$. Moreover, on classical subspaces} \[\S_\k^{cl}(V)=\bigcup_{ {\substack{t_1 \in \{0\dots,k_1-2\} \\ t_2 \in \{0,\dots,k_2-2\}}}} \left \{B_\k(\ov{t},V) + l(t) \right \}.\]

		\end{con}
		
		\added[id=h]{Conjecture \ref{con1} is just a simple generalization to any totally real field.}
		\begin{rmrk}
			\added[id=h]{Let us give a simple example to illustrate the idea behind this conjecture. Let $A$ be a non-zero $n \times n$ matrix over a $p$-adic field $K$ and let $s_1,\dots,s_n$ denote it's slopes (counted with multiplicity). Now pick an increasing function from $l:\ZZ_{\geq 0} \to \ZZ_{\geq 0}$. If we then construct the infinite block diagonal matrix $U$, whose blocks are given by $A,p^{l(1)}A,p^{l(2)}A,\dots$ then not only is $U$ compact, but it's slopes will be given by $s_1,\dots,s_n,s_1+l(1),\dots,s_n+l(1),\dots$. The above conjecture is based on a generalization of this example to the case when we don't have a single matrix $A$ determining the slopes, but several and our function $l$ is more complicated.} 
		\end{rmrk}
		
		\begin{rmrk}
			Note that, if we identify the classical space $S_{\k}(U)$ with the subspace of $L \langle X,Y \rangle$ \added[id=h]{with basis $X^{t_1}Y^{t_2}$ for $t_i \in \{0,\dots,k_i-2\}$ then the above conjecture says that associated to each basis element $X^{t_1}Y^{t_2}$, we have a finite set of $sm$-pairs $B_\k(t,X)$, with $t=(t_1,t_2)$ such that if we want to compute the slopes of $U_p$ (or $U_{\gothp_i}$) we need only the finite set $B_\k(t,U_p) $ (or $B_\k(t,U_{\ps_i})$ ) for all $t_1,t_2$ appearing in the basis of $S_\k(U)$.} \added[id=h]{This would then also give a stronger control theorem for weights near the boundary (cf. \cite{tian})}. Moreover, the set $B_\k(t,X)$ only depend on $t_1,t_2 \mod T$ and on the component of  weight space in which $\k$ lies. 
			
		\end{rmrk}

		\begin{rmrk}\label{mult}
			The conjecture above also explains the fact that the multiplicities of the slopes increasing, suggesting  that this is due to the fact that for each $x \in \ZZ_{\geq 0}$ there are $x+1$ pairs $(x_1,x_2) \in \ZZ_{\geq 0}^2$ such that $Bi(x_1,x_2)=x$, which \added[id=h]{in practice means that we have several sets of $sm$-pairs which are  equal resulting in the increase in multiplicities.} 
		\end{rmrk}
		
		\begin{rmrk}
			We note that, as a consequence of our conjecture, the slopes of Hilbert modular forms need not be given by a union of finite arithmetic progressions as we saw in \ref{rmk74}. But, it still follows that the number of slopes less than or equal to a fixed constant $\a$ is independent of the weight for weights near the boundary. From this one can show easily that, in this case, the Hilbert eigenvariety over the boundary is given by a disjoint union of rigid spaces which are finite and flat the boundary of weight space (cf. \cite[Theorem 6.22]{slbd}).
		\end{rmrk}

		\subsubsection{\textbf{ Split case}}

		The computations suggest that Conjecture \ref{splitgencon} holds with the following data.
		\begin{enumerate}
			
			\item For $F=\QQ(\sqrt{13})$,  $U=U_0(9)$, $\k=[k_1,k_2]\psi_r$ (as before) and any $t \in \ZZ_{\geq 0}^2$
			\begin{align*}
			\B_\k(\ov{t},U_p)=& \left  \{(0, 1), (1/2, 2), (1, 6), (3/2, 2), (2, 1)\right \} \\
			B_k(\ov{t},U_{\ps_i}) = &\left \{ (  0 , 3 ), (  1/2 , 6 ), (  1 , 3 )\right \}
			\end{align*}

			\item For $F=\QQ(\sqrt{17})$,  $U=U_0(8)$, $\k=[k_1,k_2]\psi_r$ and any $t \in \ZZ_{\geq 0}^2$ 
			\begin{align*}
			B_\k(\ov{t},U_p)=& \{ (  0 , 1 ), (  1/2 , 4 ) ,
			(  1 , 14 ) , (  3/2 , 4 ), (  2 , 1 ) \}\\
			B_k(\ov{t},U_{\ps_i}) = &\{  (0, 4), (1/2, 16), (1, 4)\}
			\end{align*}

		\end{enumerate}

		\subsubsection{\textbf{ Inert case}}

		Our computations suggest that Conjecture \ref{splitgencon} holds in the case where $F=\QQ(\sqrt{5})$ and $U=U_0(8\gothp_{11})$. In this case we have
		\begin{enumerate}
			\item Let $\k_1$ be any locally algebraic weight ( near the boundary with finite part $\psi_r$) in component $1$ of weight space we have $$B_{\k_1}(\ov{t},U_p)= \begin{cases} \{(2/3, 6), (1, 4), (4/3, 6)\}, & \text{if } t_1 \equiv t_2 \mod 6,\\
			\{(1/2, 4), (1, 8), (3/2, 4)\}, & \text{else.}
			\end{cases}
			$$
			
			\item Let $\k_2$ a locally algebraic weight  near the boundary with finite part $\psi_r$) in component $2$
			$$B_{\k_2}(\ov{t},U_p)= \begin{cases} \{(1/2, 4), (1, 8), (3/2, 4)\}, & \text{if } t_1 \equiv t_2 \mod 6,\\
			\{(2/3, 6), (1, 4), (4/3, 6)\}, & \text{else}.
			\end{cases}
			$$
		\end{enumerate}

		What the above suggests is that in the examples computed above the matrices $B_\k(\ov{t},V)$ in our conjecture should be the same as the matrices of $U_p,U_{\gothp_i}$ acting on classical weight $[2,2]\psi_2 \tau^{t_1-t_2}$ where $\tau$ is the character defined in \ref{tei}. This is analogous to \cite[Theorem 3.10]{bergp}. \added[id=h]{Although we note that it is possible to have levels for which the space of classical Hilbert modular forms of $[2,2]\psi_2 \tau^{t_1-t_2}$ is empty, for this reason we make the following definition:}
		
		\begin{defn}\label{minwei}
			\added[id=h]{Let $U \in D_f^\times$ have wild level $\geq \pi^2$. We define $(k_{min},r_{min})$ to be the smallest classical algebraic weight (where we order the weights lexicographically) for which there exits a finite character $\chi$ of conductor $p^2$ such that the (classical) space of Hilbert modular forms of level $U$, weight $(k_{min},r_{min})$ and character $\chi$ is non-zero. Note that the conditions imply that the locally algebraic weight associated to $(k_{min},r_{min})$ and $\chi$ is near the boundary. We call this  locally algebraic weight the {\it smallest classical weight } and denote it by $\k_{min,\chi}$}
		\end{defn}
		
		\begin{rmrk}
			\added[id=h]{In all the cases we have computed above, the smallest classical weight has algebraic part $[2,2]$.} 
		\end{rmrk}

	\added[id=h]{	From the observations above, we make the following conjecture.}
		\begin{con}\label{con2}
			Let  $V \in \{U_p,U_{\ps}\}$.	The set $B_{\k,r}(t,V)= \S_{\k_{min,\chi\tau^t}}^{cl}(V)$   where  $\k_{min,\chi\tau^t}$ for $\chi$ is as in \ref{minwei}, $\tau$ the character as in \ref{tei}  and $t \in (\ZZ/T\ZZ)^g$, with $T$ as in \ref{splitgencon}.
		\end{con}
	
	\added[id=c]{In words, this conjecture says that the slopes of $V$ can be completely determined by knowing the multiset of slopes of $V$ acting on a space of classical forms of some small predetermined weight.}

		\section{\textbf{Slopes in the centre}}\label{centrewtsp}
		
		To contrast with the computations of slopes near the boundary, we include some computations of slopes in the centre of  weight space. Here we see much less structure than near the boundary.
		
		We collect some computations of slopes for $F=\QQ(\sqrt{5})$, $p=3$ and for weights all in the same component of  weight space, which in this case means $k_1 \equiv k_2 \mod 8$. \added[id=h]{ Here again we normalize our slopes as in \ref{norma}.}
		
		\begin{nota}
			The $\dagger$ denotes overconvergent slopes. \added[id=h]{So $200^{\dagger}$ means that these are the $200$  slopes appearing in $U_p(200,\k)$ with the notation as in \ref{wow}}. The other dimensions are the dimension of the corresponding space of classical forms.
		\end{nota}
		\newpage
		\begin{longtabu}{|c| c| c| X| }
			\Xhline{2 pt} 		
			Weight& Level & Dimension & Slopes\\ \Xhline{2 pt} 
			
			$[2,2]$ & $U_0(\gothp_{11})$& $0$ & 
			\\ \hline  
			
			$[2,2]$ & $U_0(3\gothp_{11})$& $1$ &  {\fontfamily{lmr} \selectfont(0, 1), (2, 1)}
			\\ \hline  
			
			$[2,2]$ & $U_0(3\gothp_{11})$& $200^\dagger$&  {\fontfamily{lmr} \selectfont (0, 1), (1, 2), (2, 6), (3, 6), (4, 4), (9/2, 4), (5, 8), (11/2, 4), (6, 30), (13/2, 4), (7, 4), (8, 6), (17/2, 4), (9, 6), (10, 22), (11, 10), (23/2, 4), (12, 19), (13, 8), (40/3, 3), (27/2, 6), (14, 30), (29/2, 4), (15, 3), (16, 2)}
			\\ \Xhline{2 pt}

			$[4,4]$ & $U_0(\gothp_{11})$& $1$ & $(0,1)$
			\\ \hline  
			
			$[4,4]$ & $U_0(3\gothp_{11})$& $18$ & {\fontfamily{lmr} \selectfont (0, 1), (2, 16), (6, 1)}
			\\ \hline  
			
			$[4,4]$ & $U_0(3\gothp_{11})$& $200^\dagger$&  {\fontfamily{lmr} \selectfont(0, 1), (2, 16), (3, 2), (4, 6), (5, 4), (6, 24), (19/3, 6), (7, 10), (8, 20), (9, 6), (10, 16), (11, 20), (12, 28), (13, 16), (14, 12), (15, 7), (16, 3), (53/3, 3)}
			\\ \Xhline{2 pt}

			$[6,6]$ & $U_0(\gothp_{11})$& $5$ &  {\fontfamily{lmr} \selectfont (0,1),(1,2),(2,2)}
			\\ \hline  
			
			$[6,6]$ & $U_0(3\gothp_{11})$& $50$ &   {\fontfamily{lmr} \selectfont(0, 1), (1, 2), (2, 2), (4, 40), (8, 2), (9, 2), (10, 1)}
			\\ \hline  
			
			$[6,6]$ & $U_0(3\gothp_{11})$& $200^\dagger$&  {\fontfamily{lmr} \selectfont(0, 1), (1, 2), (2, 2), (4, 40), (5, 2), (6, 2), (13/2, 4), (7, 20), (8, 6), (17/2, 4), (9, 10), (19/2, 4), (10, 12), (21/2, 4), (11, 30), (23/2, 4), (12, 17), (25/2, 2), (13, 15), (14, 8), (15, 7), (31/2, 2), (16, 2)}
			\\ \Xhline{2 pt}

			$[14,6]$ & $U_0(\gothp_{11})$& $13$ &   {\fontfamily{lmr} \selectfont(0,1), (1,2), (2,2), (3,6), (4,2)}
			\\ \hline

			$[14,6]$ & $U_0(3\gothp_{11})$& $130$ & {\fontfamily{lmr} \selectfont (0, 1),
				(1, 2),
				(2, 2),
				(3, 6),
				(4, 2),
				(8, 104),
				(14, 2),
				(15, 6),
				(16, 2),
				(17, 2),
				(18, 1)}
			\\ \hline  
			
			$[14,6]$ & $U_0(3\gothp_{11})$& $200^\dagger$& {\fontfamily{lmr} \selectfont (0, 1),
				(1, 2),
				(2, 2),
				(3, 6),
				(4, 2),
				(5, 1),
				(6, 1),
				(13/2, 2),
				(7, 3),
				(15/2, 2),
				(8, 104),
				(9, 8),
				(10, 3),
				(21/2, 2),
				(11, 12),
				(23/2, 14),
				(47/4, 4),
				(12, 4),
				(25/2, 2),
				(13, 3),
				(14, 2),
				(15, 2),
				(16, 3),
				(33/2, 2),
				(17, 6),
				(18, 7)}
			\\ \Xhline{2 pt} 	
			\caption{Centre slopes for $p=3$.}
		\end{longtabu}

		\begin{obs}
			The first observation is that in this case, the slopes are not appearing as unions of arithmetic progressions. Moreover, there are many non-integer slopes, which is in contrast to many examples over $\QQ$.
		\end{obs}

		In the table below we work in $\QQ(\sqrt{5})$, with $p=2$ and level $\K_0(2 \gothp_{11})$, where $\ps_{11} |11$.

		\begin{longtabu}{|c |c| c |X|}
			\Xhline{2 pt}
			Level &	Weight & Dimension & Slopes \\  \Xhline{2 pt}
			
			$\K_0(\gothq_{11})$&		$[2,2]$ & $0$  &   
			
			\\ \hline

			$\K_0(2 \gothq_{11})$&	$[2,2]$ & $1$  &   
			$(2, 1)$
			\\ \hline

			$\K_0(2 \gothq_{11})$&	$[2,2]$ & $200^\dagger$ &   
			
			$ (  2 , 5 ) $,
			$ (  3 , 2 ) $,
			$ (  4 , 10 ) $,
			$ (  6 , 3 ) $,
			$ (  7 , 4 ) $,
			$ (  8 , 33 ) $,
			$ (  10 , 4 ) $,
			$ (  11 , 2 ) $,
			$ (  12 , 3 ) $,
			$ (  25/2 , 4 ) $,
			$ (  13 , 12 ) $,
			$ (  14 , 10 ) $,
			$ (  44/3 , 6 ) $,
			$ (  15 , 8 ) $,
			$ (  46/3 , 6 ) $,
			$ (  31/2 , 8 ) $,
			$ (  16 , 40 ) $,
			$ (  49/3 , 3 ) $,
			$ (  33/2 , 4 ) $,
			$ (  50/3 , 3 ) $,
			$ (  17 , 14 ) $,
			$ (  35/2 , 2 ) $,
			$ (  18 , 7 ) $,
			$ (  19 , 3 ) $,
			$ (  20 , 2 ) $,
			$ (  41/2 , 2 ) $
			\\ \Xhline{2 pt} 
			
			$\K_0(\gothq_{11})$&			$[4,4]$ & $1$ &   
			$(2, 1)$	 	
			\\ \hline 		
			$\K_0(2 \gothq_{11})$&		$[4,4]$ & $9$ &   
			$(2, 8), (4, 1)$	 	
			\\ \hline

			$\K_0(2 \gothq_{11})$&	$[4,4]$ & $200^\dagger$  &   
			$ (  2 , 8 ) $,
			$ (  4 , 1 ) $,
			$ (  5 , 4 ) $,
			$ (  16/3 , 6 ) $,
			$ (  6 , 6 ) $,
			$ (  7 , 2 ) $,
			$ (  8 , 7 ) $,
			$ (  17/2 , 12 ) $,
			$ (  9 , 4 ) $,
			$ (  10 , 21 ) $,
			$ (  11 , 2 ) $,
			$ (  12 , 6 ) $,
			$ (  25/2 , 4 ) $,
			$ (  13 , 6 ) $,
			$ (  27/2 , 4 ) $,
			$ (  14 , 3 ) $,
			$ (  15 , 22 ) $,
			$ (  16 , 8 ) $,
			$ (  33/2 , 4 ) $,
			$ (  17 , 13 ) $,
			$ (  35/2 , 24 ) $,
			$ (  18 , 21 ) $,
			$ (  19 , 6 ) $,
			$ (  39/2 , 2 ) $,
			$ (  20 , 3 ) $,
			$ (  21 , 1 ) $
			\\ \Xhline{2 pt}

			$\K_0(\gothq_{11})$&		$[4,2]$ & $1$ &
			$(1, 1)$
			\\ \hline

			$\K_0(2 \gothq_{11})$&		$[4,2]$ & $3$ &
			$(1, 2), (3, 1)$
			\\ \hline

			$\K_0(2 \gothq_{11})$&	$[4,2]$ & $200^\dagger$ &
			$ (  1 , 2 ) $,
			$ (  2 , 1 ) $,
			$ (  5/2 , 2 ) $,
			$ (  3 , 1 ) $,
			$ (  4 , 5 ) $,
			$ (  14/3 , 3 ) $,
			$ (  5 , 5 ) $,
			$ (  6 , 1 ) $,
			$ (  13/2 , 2 ) $,
			$ (  7 , 5 ) $,
			$ (  8 , 11 ) $,
			$ (  9 , 24 ) $,
			$ (  19/2 , 2 ) $,
			$ (  10 , 1 ) $,
			$ (  11 , 9 ) $,
			$ (  12 , 4 ) $,
			$ (  13 , 2 ) $,
			$ (  14 , 19 ) $,
			$ (  29/2 , 2 ) $,
			$ (  15 , 7 ) $,
			$ (  46/3 , 3 ) $,
			$ (  16 , 12 ) $,
			$ (  49/3 , 6 ) $,
			$ (  33/2 , 16 ) $,
			$ (  50/3 , 3 ) $,
			$ (  67/4 , 4 ) $,
			$ (  17 , 21 ) $,
			$ (  35/2 , 6 ) $,
			$ (  18 , 9 ) $,
			$ (  37/2 , 4 ) $,
			$ (  19 , 5 ) $,
			$ (  20 , 1 ) $,
			$ (  21 , 2 ) $
			\\ \Xhline{2 pt}

			$\K_0(\gothq_{11})$&$[6,2]$ & $1$ & 
			$(1, 1)$
			\\ \hline

			$\K_0(2 \gothq_{11})$&$[6,2]$ & $5$ & 
			$(1, 1), (2, 3), (5, 1)$
			\\ \hline

			$\K_0(2 \gothq_{11})$&	$[6,2]$ & $200^\dagger$ & 
			$ (  1 , 1 ) $,
			$ (  2 , 3 ) $,
			$ (  3 , 2 ) $,
			$ (  10/3 , 3 ) $,
			$ (  4 , 1 ) $,
			$ (  9/2 , 4 ) $,
			$ (  5 , 4 ) $,
			$ (  17/3 , 3 ) $,
			$ (  6 , 2 ) $,
			$ (  7 , 4 ) $,
			$ (  15/2 , 4 ) $,
			$ (  8 , 2 ) $,
			$ (  17/2 , 2 ) $,
			$ (  9 , 17 ) $,
			$ (  19/2 , 2 ) $,
			$ (  10 , 12 ) $,
			$ (  21/2 , 2 ) $,
			$ (  11 , 6 ) $,
			$ (  12 , 12 ) $,
			$ (  13 , 5 ) $,
			$ (  14 , 2 ) $,
			$ (  15 , 26 ) $,
			$ (  31/2 , 2 ) $,
			$ (  16 , 7 ) $,
			$ (  33/2 , 2 ) $,
			$ (  17 , 18 ) $,
			$ (  35/2 , 16 ) $,
			$ (  53/3 , 6 ) $,
			$ (  18 , 17 ) $,
			$ (  19 , 7 ) $,
			$ (  39/2 , 2 ) $,
			$ (  20 , 4 ) $		
			\\\Xhline{2 pt}

			\caption{Centre slopes for $p=2$.}	
		\end{longtabu}

		\begin{rmrk}
			We can again see that in this case there is much less structure to the slopes. In particular, they do not appear to be unions of arithmetic progressions and their structure is not obviously different from the regular case. Moreover, if one make the naive extension of the definitions of $\Gamma_0$-regular and $\Gamma_0$-irregular as in \cite{qasl}, then in the above examples $p=3$ would be regular and $p=2$ would be regular, but there does not appear to be any difference in the structure of the slopes in these cases.
		\end{rmrk}

	\end{section}

	\bibliographystyle{alpha}
	
	\bibliography{bibliog}

\end{document}